\documentclass[12pt]{amsart}
\usepackage{amsmath, amsthm, amssymb}
\usepackage{mathtools}
\usepackage{todonotes}
\usepackage{comment}
\usepackage{dsfont}
\usepackage{thm-restate}
\usepackage{pgfplots}
\usepackage{tikz}
\usetikzlibrary{math}
\usetikzlibrary{arrows} 
\usepackage{tikz-3dplot}
\usetikzlibrary{shapes,calc,positioning}
\usepackage{amsthm}
\usepackage{caption}
\usepackage{xcolor}
\usepackage{hyperref}
\usepackage{multirow}
\usepackage{cleveref}

\usepackage{float}

\usepackage[paper=a4paper, margin=3cm]{geometry}
\usepackage{array}
\usepackage{tabularx}
\usepackage[normalem]{ulem}

\usepackage{hyperref}
\hypersetup{
    unicode=false,          
    pdftoolbar=true,        
    pdfmenubar=true,        
    pdffitwindow=false,     
    pdfstartview={FitH},    
    pdftitle={},    
    pdfauthor={; Monin, L.},     
    pdfkeywords=  
    pdfnewwindow=true,      
    colorlinks=true,       
    linkcolor=magenta,          
    citecolor=cyan,        
    filecolor=magenta,      
    urlcolor=cyan           
}

\numberwithin{equation}{section}
\newtheorem{theorem}{Theorem}
\newtheorem{proposition}[theorem]{Proposition}
\newtheorem{lemma}[theorem]{Lemma}
\newtheorem{cor}[theorem]{Corollary}

\theoremstyle{definition}
\newtheorem{definition}[theorem]{Definition}
\newtheorem{example}[theorem]{Example}
\newtheorem{remark}[theorem]{Remark}

\newtheorem{conj}[theorem]{Conjecture}
\newtheorem{quest}{Question}

\numberwithin{theorem}{section}
\numberwithin{remark}{section}
\numberwithin{example}{section}
\numberwithin{definition}{section}
\numberwithin{proposition}{section}
\numberwithin{assumption}{section}
\numberwithin{lemma}{section}
\numberwithin{cor}{section}

\renewcommand{\P}{\mathbb{P}}
\newcommand{\R}{\mathbb{R}}
\newcommand{\Q}{\mathbb{Q}}
\newcommand{\C}{\mathbb{C} }
\newcommand{\Z}{\mathbb{Z}}
\newcommand{\N}{\mathbb{N}}

\newcommand{\chern}{\operatorname{c} }
\renewcommand{\ss}{\operatorname{s} }
\newcommand{\EE}{\mathcal{E} }
\newcommand{\VV}{\mathcal{V} }

\newcommand{\FF}{\mathcal{F} }
\newcommand{\LL}{\mathcal{L} }

\renewcommand{\AA}{\mathcal{A}}

\newcommand{\OO}{\mathcal{O}}

\newcommand{\conv}{\operatorname{conv} }
\newcommand{\rank}{\operatorname{rank} }

\newcommand{\newt}{{\mathrm{Newt}}}

\newcommand{\init}{{\mathrm{init}}}
\newcommand{\vol}{{\mathrm{Vol}}}
\newcommand{\mvol}{{\mathrm{MV}}}
\newcommand{\cay}{{\mathrm{Cay}}}
\newcommand{\nc}{{\Sigma}}

\newtheorem{thmx}{Theorem}

\newcommand\restr[2]{{
  \left.\kern-\nulldelimiterspace 
  #1 
  \vphantom{\big|} 
  \right|_{#2} 
  }}

\title{The algebraic degree of sparse polynomial optimization}
\author{Julia Lindberg}
\address{Julia Lindberg \\ 
University of Texas-Austin}
\email {julia.lindberg@math.utexas.edu}\urladdr{https://sites.google.com/view/julialindberg/home}

\author{Leonid Monin}
\address{Leonid Monin \\
\'Ecole Polytechnique F\'ed\'erale de Lausanne}
\email {leonid.monin@epfl.ch}
\urladdr{http://www.math.toronto.edu/lmonin/}

\author{Kemal Rose}
\address{Kemal Rose \\
KTH Royal Institute of Technology}
\email {kemalr@kth.se}\urladdr{https://kemalrose.github.io/}

\date{}

\newpage
\pgfplotsset{compat=1.18}
\begin{document}

\begin{abstract}
We study a broad class of polynomial optimization problems
whose constraints and objective functions exhibit sparsity patterns.
We give two characterizations of the number of critical points to these problems, one as a mixed volume and one as an intersection product on a toric variety.
As a corollary, we obtain a convex geometric interpretation of polar degrees, a classical invariant of algebraic varieties, as well as Euclidean distance degrees.
Furthermore, we
prove the BKK generality of Lagrange systems in many instances. 
\end{abstract}
\maketitle
{\tiny\textbf{Keywords: } sparse polynomial optimization, toric varieties, algebraic degree, ED degree, polar degree}

{\tiny\textbf{MSC2020:  } 90C26, 14M25, 14Q20, 52B20}

\section{Introduction}

We consider polynomial optimization problems of the form:
\begin{align*}
    \min_{y \in \mathbb{R}^n} \ f_0(y) \quad \text{subject to} \quad f_1(y) = 0,\ldots,f_m(y) = 0 \label{POP} \tag{POP}
\end{align*}
where $f_i \in \mathbb{R}[y_1,\ldots,y_n]$. Polynomial programs have broad modeling power and are found in applications such as signal processing, combinatorial optimization, power systems engineering, and more \cite{tan2001the,poljak1995a,molzahn2019a}. In general, these problems are NP-hard to solve \cite{vavasis1990quadratic} but there has been much work in
recent decades studying various solution techniques for \eqref{POP}. Inspired by recent improvements in numerical methods for polynomial system solving, in this work we focus on solving the critical point equations for \eqref{POP}.

When the first-order optimality conditions hold, there are finitely many complex critical points to \eqref{POP}. For a specified objective function $f_0$ and fixed constraints $f_1,\ldots,f_m$ we denote $\mathbf{F} = (f_0, \ldots,f_m )$. The number of complex critical points of $\mathbf{F}$ is called the \emph{algebraic degree} of $\mathbf{F}$. While \eqref{POP} is a real optimization problem, we consider complex critical points since for polynomials $\mathbf{F}$ with fixed monomial support, the algebraic degree is generically\footnote{By generic, we mean generic with respect to the Zariski topology. See \Cref{rem:generic} for a detailed explanation.} constant. When the algebraic degree of \eqref{POP} is finite, 
the coordinates of an optimal solution of \eqref{POP} can be expressed as the roots of 
polynomial functions of the coefficients of $\mathbf{F}$. 
Under sufficient genericity conditions these polynomial functions are all of the same degree and the algebraic degree of $\mathbf{F}$ has an additional interpretation as the degree of these polynomials.

The algebraic degree of $\mathbf{F}$ also has practical implications for optimization algorithms as it gives an upper bound on the number of critical points of \eqref{POP}. This gives a bound on
the number of local extrema at which local optimization algorithms can terminate.

When $\mathbf{F}$ consists of generic polynomials with full monomial support, a formula for the algebraic degree of $\mathbf{F}$
was given in \cite{MR2507133}. This was then specialized for many classes of convex polynomial optimization problems in \cite{MR2496496} and \cite{MR2546336}. When the objective function of \eqref{POP} is the Euclidean distance function, i.e. $f_0 = \lVert y - u \rVert_2^2$ for a generic point $u \in \mathbb{R}^n$, 
the algebraic degree of $\mathbf{F}$
is called the \emph{ED degree} of $ (f_1,\ldots,f_m )$. The study of ED degrees began with \cite{EDdegree2016}. Follow up work then studied the ED degree of real algebraic groups \cite{baaijensRealAlgGroups}, Fermat hypersurfaces \cite{leeFermat}, orthogonally invariant matrices \cite{drusvyatskiyOrthogonally}, smooth complex projective varieties \cite{aluffiEDComplex}, the multiview variety \cite{maximMultiview}, hypersurfaces \cite{breiding2020euclidean} and when the data $u$ and polynomials $f_1,\ldots,f_m$ are not generic \cite{maximDefect}.

A related problem coming from statistics considers the likelihood objective function, $f_0 = y_1^{u_1}\cdots y_n^{u_n}$. 
This problem is called \emph{maximum likelihood estimation} and 
the number of complex critical points to \eqref{POP} when $f_0= y_1^{u_1}\cdots y_n^{u_n}$ for generic $u \in \mathbb{R}^n$ is called the \emph{ML degree} of $ (f_1,\ldots,f_m )$.  Relationships between ML degrees and Euler characteristics as well as the ML degree of various statistical models have been studied in \cite{MR2230921, HostenSolving,MR3103064,MR3907355, DM2021,MR4219257,manivel2023complete}.

Inspired by recent results on the ED and ML degrees of sparse polynomial systems \cite{breiding2020euclidean,lindberg2021the}, we study the algebraic degree of \eqref{POP} when each polynomial $f_i \in \mathbb{R}[y_1,\ldots,y_n]$ is assumed to be a \emph{sparse polynomial} with generic coefficients. Given an optimization problem of the form \eqref{POP} where $\mathbf{F}$ is a list of sparse polynomial equations with generic coefficients, the \emph{Lagrangian} of $\mathbf{F}$ is
$$\Phi_F(\lambda, y) := f_0 - \sum_{i=1}^m \lambda_i f_i.$$
We consider the \emph{Lagrange system} of $\mathbf{F}$, namely $\mathbf{L}_{\mathrm{F}} = (f_1,\ldots,f_m,\ell_1,\ldots,\ell_n)$, where
\[ \ell_j = \frac{\partial}{\partial y_j} \left( f_0 - \sum_{i=1}^m \lambda_i f_i \right).\]

Analogous to the \emph{algebraic degree of polynomial optimization} from \cite{ MR2507133},
we generalize the common term to the
\emph{algebraic degree of sparse polynomial optimization}. It is the
number of complex critical points of $f_0$ restricted to
$\mathcal{V}(f_1,\ldots,f_m)$, where each $f_i \in \R[y_1,\ldots,y_n]$ is a \emph{sparse} polynomial. When all critical points are smooth, it is:
\begin{equation}
    \label{eq:degree of sparse pol optimization}
    \# \VV \left( \mathbf{L}_{\mathrm{F}}  \right) =
    \# \left \{ (y, \lambda) \in \C^n \times \C^m: \quad  0 =
    f_1 = \cdots = f_m = 
    \ell_1 = \cdots = \ell_n \right \}.
    \end{equation}

There exist classical results in algebraic geometry bounding 
the number of isolated solutions to a square polynomial system. A result of B\'ezout says that $\#\mathcal{V}(\mathbf{L}_{\mathrm{F}})$ is bounded above by the product of the degrees of the polynomials in $\mathbf{L}_{\mathrm{F}}$. 
If $\deg(f_i)~=~d_i$ and $\deg(\ell_j) = h_j$, $0 \leq i \leq m, 1 \leq j \leq n$, B\'ezout's bound reduces to $d_1 \cdots d_m\cdot h_1 \cdots h_n$ where $h_j \leq \max_{i \in [m]} \{ d_0 - 1, d_i\}$. Work of Nie and Ranestad refined this bound and showed that 
\[ \# \mathcal{V}(\mathbf{L}_{\mathrm{F}}) \leq d_1 \cdots d_m \cdot D_{n-m}(d_0 - 1,\ldots, d_m - 1) \]
where $D_{r}(n_1,\ldots, n_k) = \sum_{i_1 + \cdots + i_k = r} n_1^{i_1} \cdots n_k^{i_k}$
\cite{MR2507133}.
While this bound is generically tight when $\mathbf{F}$ consists of polynomials with full monomial support, the following example shows that it can be 
a strict upper bound
when $\mathbf{F}$ is sparse.

\begin{example}\label{ex1}
Consider the following optimization problem:
\begin{align}
    \min_{y \in \R^n} \ c^T y \quad \text{subject to} \quad f = \alpha_1 y_1^3 + \sum_{j=2}^{n-1} \alpha_j y_j^2 + \alpha_n y_n = 1, \label{opt ex}
\end{align}
where $c, \alpha \in \R^n$ are generic parameters.
The corresponding Lagrange system is given by $\mathbf{L}_{\mathrm{F}} = (\ell_1,\ldots,\ell_n, f ) $ where
\[ \ell_1 = c_1 - 3\lambda \alpha_1 y_1^2, \quad \ell_n = c_n - \alpha_n \lambda, \quad \ell_j = c_i - 2 \lambda \alpha_j y_j, \ 2 \leq j \leq n-1.  \]
B\'ezout's bound gives that $\#\mathcal{V}(\mathbf{L}_{\mathrm{F}}) \leq 9 \cdot 2^{n-2}$, which is worse than the Nie-Ranestad bound $\#\mathcal{V}(\mathbf{L}_{\mathrm{F}}) \leq 3 \cdot D_{n-1}(0,2) = 3 \cdot 2^{n-1}$. In this case, the sparsity of the Lagrange equations allows one to solve the Lagrange system explicitly.
One can see that for generic values of $c$ and $\alpha$, $\# \mathcal{V}(\mathbf{L}_{\mathrm{F}}) = 2$. 
\end{example}

Motivated by the previous example, we are interested in obtaining tighter bounds for the algebraic degree of sparse polynomial optimization problems.
We focus on the following questions. 
\begin{quest}
\label{quest: nr critical pts}
How many smooth critical points does \eqref{POP}
have for sparse $\mathbf{F}$?
\end{quest}

The motivation for \Cref{quest: nr critical pts} is that if we know how many critical points \eqref{POP} has, and we find them all, then we can globally solve \eqref{POP}. 
Currently, the only way to \emph{provably} find all smooth critical points is to find all complex zeros of $\mathbf{L}_{\mathrm{F}}$.

The field of computational algebraic geometry has traditionally been associated with symbolic computations based on Gröbner bases. Recent developments in numerical frameworks, such as homotopy continuation \cite{bertini}, provide algorithms that are able to solve problems intractable with symbolic methods. Moreover, numerical algorithms can not only provide the floating point approximation of a solution, but also certify that a given approximation represents a unique solution, and provide guarantees for correctness \cite{breiding2021certifying,Rump1999,LeeM2}.
Therefore, numerical computation can be used to prove lower bounds on $\# \VV (\mathbf{L}_{\mathrm{F}})$.
However, ensuring the absence of additional solutions requires establishing an upper bound, which can be achieved through intersection theory.

Such an intersection theoretic bound for sparse polynomial systems was given by the celebrated Bernstein-Kouchnirenko-Khovanskii (BKK) theorem. The BKK theorem \cite{bernshtein1979the} relates the number of $\C^* := \C \backslash \{0\}$ zeros of a system of polynomial equations to the mixed volume of their corresponding Newton polytopes (see \Cref{sec: bkk}). While the BKK bound is generically tight, we note that
the coefficients of the system $\mathbf{L}_{\mathrm{F}}$ are linearly dependent, so
a priori it is not clear that the system $\mathbf{L}_{\mathrm{F}}$
has the expected number of solutions.
This motivates the second question guiding this work.
\begin{quest}
\label{quest: is Lagragian BKK general?}
Is the Lagrange system of \eqref{POP} BKK general?
\end{quest}

An affirmative answer to \Cref{quest: is Lagragian BKK general?} would show that \emph{polyhedral homotopy} algorithms are optimal for finding all complex critical points for generic instances of \eqref{POP} in the sense that for every solution to $\mathbf{L}_{\mathrm{F}} = 0$, exactly one \emph{homotopy path} is tracked. For more details on polyhedral homotopy continuation see \cite{huber1995a}.

A second consequence of \Cref{quest: is Lagragian BKK general?} has to do with the complexity of algorithms in real algebraic geometry. A fundamental question in real algebraic geometry is to find a point on each connected component of a real algebraic variety, $Y$. In \cite{hauenstein2013numerically}, Hauenstein showed that when $f_0$ is the Euclidean distance function, the zeros of the Lagrange system of $f_0$ restricted to $Y$ cover all connected components. 
He then used this result to suggest a homotopy algorithm that would find a point on each connected component of $Y$.
In Hauenstein's homotopy algorithm, he partitions the homotopy paths into two sets, $E$ and $E_1$, where $E$ denotes the set of homotopy paths that converge and $E_1$ denotes the set of homotopy paths that diverge. 
By showing that polyhedral homotopy algorithms are optimal for a large class of Euclidean distance optimization problems, we would
further the analysis of Hauenstein's algorithm by showing that the set $E_1$ is empty.

Finally, we note that understanding BKK exactness of non generic polynomial systems is of increasing interest in the applied algebraic geometry community. In this field, researchers study systems of polynomial equations coming from a diverse set of fields including biology, optimization, data science and power systems engineering. Recent results also study BKK exactness in these settings \cite{coons2023mixed,breiding2020euclidean,lindberg2021the,lindberg2023estimating,chen2022typical}.

\subsection*{Contribution}
Our contribution consists of multiple results 
determining the algebraic degree of generic sparse polynomial programs.
First, we show in \Cref{thm: B}, that the answer to \Cref{quest: is Lagragian BKK general?} is positive for a wide class of sparse polynomial programs having strongly admissible monomial support (see \Cref{def: strongly admissible point configuration}). In particular, our results show that the BKK bound is tight for \Cref{ex1}.  Further, in \Cref{cor: delete mons} we show that algebraic degrees of generic sparse polynomial programs are determined by the Newton polytopes of $\mathbf{F}$. \Cref{cor: delete mons} also has algorithmic implications which were studied in \cite{rose2023polyhedral}.

As a corollary, we prove an analogous result to that in \cite{breiding2020euclidean}, and show that the ED degree of a variety defined by polynomials with strongly admissible support is equal to the BKK bound of its corresponding Lagrange system (\Cref{cor:EDdegree}). We also prove similar results for (the sum of) polar degrees (\Cref{cor:mvol and polar}), giving the first convex geometric interpretation of this algebraic invariant.

In \Cref{thm: C} we loosen the assumption on $\mathbf{F}$ of strongly admissible monomial support and provide a different formula for the algebraic degree of $\mathbf{F}$ 
when $\mathbf{F}$ is in a larger family of sparse polynomial programs.
Our main tool here is Porteous' formula which computes the fundamental class of the degeneracy locus of a morphism between two vector bundles as a polynomial of their Chern classes. Using Porteous' formula, \Cref{thm: C} expresses the algebraic degree in terms of the intersection theory of a certain toric compactification of $(\C^*)^n$. The formula for the algebraic degree in \Cref{thm: C} can be expressed as a (not necessarily positive) linear combination of mixed volumes. However, the explicit connection to the mixed volume of the Lagrange system is still mysterious.

\section{Preliminaries and notation}\label{sec: prelims}
\subsection{Sparse polynomials and polyhedral geometry}\label{sec: bkk}
We specify a \emph{sparse polynomial} $f \in \mathbb{C}[y_1,\ldots,y_n]$ by its monomial support, $\AA \subset \mathbb{N}^n := \Z_{\geq 0}^n$. Specifically, for a finite subset $\AA \subset \mathbb{N}^n$, we write 
\[ f = \sum_{\alpha \in \AA} c_{\alpha} y^\alpha \]
where $y^{\alpha} := y_1^{\alpha_1} \cdots y_n^{\alpha_n}$ and $c_\alpha \in \mathbb{C}$. 
Given a sparse polynomial $f \in \mathbb{C}[y_1,\ldots, y_n]$ with monomial support $\AA$, the \emph{Newton polytope of $f$} is defined as the convex hull of $\AA$. Explicitly, the Newton polytope of $f$ is 
\[\newt(f) = \conv \{ \alpha \ : \ \alpha \in \AA \}. \]
A \emph{sparse polynomial system} $\mathbf{F} =(f_0,\ldots,f_m)$ is defined by a tuple $\AA = (\AA_0,\ldots, \AA_m)$ where $f_i = \sum_{\alpha \in \AA_i} c_{\alpha,i} y^{\alpha} \in \mathbb{C}[y_1,\ldots,y_n]$.

\begin{remark}\label{rem:generic}
In this paper, we consider \emph{generic} sparse polynomial systems.
A statement holds for a generic sparse polynomial system if
it holds for all systems $\mathbf{F} = (f_0,\ldots, f_m)$
whose coefficients $\{c_{\alpha,i} \ : \alpha \in \mathcal{A}_i, \  0 \leq i \leq m \}$
lie in some nonempty Zariski open subset of the space $\C^{\AA_0} \times \cdots \times \C^{\AA_m}$ of coefficients. This means that the non-generic behavior occurs on a set of measure zero in the space $\C^{\AA_0} \times \cdots \times\C^{\AA_m}$.
\end{remark}

Given polytopes $P_1,\ldots,P_n \subset \mathbb{R}^n$, the \emph{mixed volume} of $P_1,\ldots,P_n$ is the coefficient of the monomial $\lambda_1\cdots\lambda_n$ in the volume polynomial 
\[ \vol_n(\lambda_1 P_1 + \ldots + \lambda_n P_n), \]
where $P + Q = \{p + q \ : \ p \in P, \ q \in Q\}$ is the Minkowski sum and $\vol_n$ is the standard $n$-dimensional Euclidean volume.

In a series of celebrated results \cite{bernshtein1979the,kouchnirenko1976polyedres,khovanskii1978newton} the connection between the number of solutions over $\mathbb{C}^*$ to a system of sparse polynomial equations and the underlying convex geometry of the polynomials was made. 

\begin{theorem}[BKK Bound \cite{bernshtein1979the,kouchnirenko1976polyedres,khovanskii1978newton}] \label{thm: BKK}
Let $\mathbf{F} =(f_1,\ldots,f_n) \subset \mathbb{C}[y_1,\ldots,y_n]$ be a sparse polynomial system with $\eta$ isolated
solutions in $(\mathbb{C}^*)^n$, counted with multiplicity and let $P_i = \newt(f_i)$. Then 
\[ \eta \leq \mvol(P_1,\ldots,P_n).\] Moreover, if the coefficients of $\mathbf{F}$ are generic then $\eta = \mvol(P_1,\ldots,P_n)$.
\end{theorem}

If for a sparse polynomial system $\mathbf{F}$ the BKK bound holds with equality, we say $\mathbf{F}$ is \emph{BKK general}. Bernstein gave explicit degeneracy conditions under which the above inequality is tight by considering the initial systems of $\mathbf{F}$ \cite{bernshtein1979the}.

Given a polytope $P \subset \mathbb{R}^n$ and a vector $w \in \mathbb{Z}^n \backslash \{0\}$, let $P_w$ denote the \emph{face exposed} by $w$. Specifically, 
\[ P_w = \{ v \in P \ : \ \langle v, w \rangle \leq \langle z, w \rangle \ \forall z \in P \}. \]
For a sparse polynomial $f$ we call 
\[ \init_w(f) = \sum_{ \alpha \in (\newt(f))_w} c_{\alpha} y^\alpha \]
the \emph{initial polynomial} of $f$ with respect to $w$. For a sparse polynomial system $\mathbf{F}$, we denote $\init_w(\mathbf{F}) = (\init_w(f_1),\ldots,\init_w(f_n))$.

\begin{theorem}\cite[Theorem 2]{bernshtein1979the}\label{thm: bkk 2}
Let $\mathbf{F} =(f_1,\ldots,f_n) \subset \mathbb{C}[y_1,\ldots,y_n]$ be a sparse polynomial system with $\eta$ isolated $\mathbb{C}^*$ solutions counted with multiplicity and let $P_i = \newt(f_i)$. All $\mathbb{C}^*$ solutions of $F(y) = 0$ are isolated and $\eta = \mvol(P_1,\ldots,P_n)$ if and only if for every $w \in \mathbb{Z}^n \backslash \{0\}$, $\init_w(\mathbf{F})$ has no $\mathbb{C}^*$ solutions.
\end{theorem}

\Cref{thm: BKK} and \Cref{thm: bkk 2} demonstrate the intimate connection between solutions to systems of polynomial equations and polyhedral geometry. In the remainder of this section, we define a few more objects that are helpful when using this connection. Given $\AA = (\AA_0,\ldots,\AA_m) \in \N^n$ we define the \emph{Cayley polytope} of $\AA$ as

\[ 
\cay (\AA) = \conv \left(\{ (z, e_{i}) \ : z \in \AA_i, i = 0, \dots, m \} \right)\subset \R^{n+m} 
\]
where $e_i$ is the $i$-th standard basis vector of $\R^m$ and $e_0$ is the vector of all zeroes. Similarly, for a sparse polynomial system $\mathbf{F}$ with support $\AA = (\AA_0,\ldots,\AA_m)$, we define $\cay(\mathbf{F}) = \cay(\AA)$.

For a face of a convex polytope, $F \subset P \subset \R^n$, the \emph{normal cone} of $F$ is the set of linear functionals which achieve their minimum on $F$:
\[ \sigma (F) = \{c \in \R^n \ : \ \langle c,x \rangle \leq \langle c,z \rangle, \ \forall x \in F, \ z \in P \}. \]
The normal cones of each face of $P$ form a fan, denoted $\nc(P) \subset \R^n$.

Finally, in this paper we consider the operation of taking the ``partial derivative" of a polytope which we define as follows. 
Let $P\subset \R_{\geq 0}^n$ be a polytope contained in the positive orthant. Then
\[
\partial_j P= (P-e_j)\cap \R_{\geq 0}^n = \{ \alpha - e_j \ : \ \alpha \in P , \  \alpha_j \geq 1\},
\]
where $e_j$ is the $j-$th standard basis vector of $\R^n$.

The definition of $\partial_j P$ is motivated by the partial differentiation operation of a polynomial $f \in \R[y_1,\ldots,y_n]$ with $\newt(f)=P$. Indeed, one always has $\newt(\frac{\partial}{\partial y_j} f)\subset\partial_j \newt(f)$. However, $\newt(\frac{\partial}{\partial y_j} f)$ and $\partial_j \newt(f)$ can be quite different. For instance, let $f(y)=1+y^3$, then $\newt(f)=[0,3]$ and $\partial\newt(f)=[0,2]$. However, $\frac{\partial}{\partial y} f= 3y^2$ and $\newt(\frac{\partial}{\partial y} f)= \{2\}$. In general, even if $P$ is integral, the polytope $\partial_i P$ does not have to be integral.

For a polynomial $f=\sum_{\alpha \in \AA} c_\alpha y^\alpha$ having full monomial support with respect to $\newt(f)$ (i.e. $c_\alpha \ne 0$ for any $\alpha \in \newt(f)\cap\N^n$) the two constructions are connected in the following way: 
\[
\newt\left(\frac{\partial}{\partial y_j} f\right) = \conv \{\partial_j \newt(f)\cap \N^n\}.
\]
In particular, if $f$ is a degree $d$ polynomial with all monomials of degree $\leq d$ having non-zero coefficients, then 
\[
\newt\left(\frac{\partial}{\partial y_j} f\right) = \partial_j \newt(f).
\]

\subsection{Toric varieties}\label{sec: toric}
\Cref{thm: BKK} is a statement about intersection theory on toric varieties. A \emph{toric variety} $X$ is an irreducible variety such that $(\C^*)^n$ is a Zariski open subset of $X$ and the action of $(\C^*)^n$ on itself extends to an action of $(\C^*)^n$ on $X$. We can also associate normal toric varieties with polyhedral fans.

Let $\sigma \in \R^n$ be a rational polyhedral cone which does not contain any vector subspace and denote 
\[ S_\sigma = \sigma^{\vee} \cap \Z^n\]
where $\sigma^{\vee} = \{ y \in \R^n \ : \ \langle y,x \rangle \geq 0 \ \forall x \in \sigma \}$ is the dual cone of $\sigma$. Then the \emph{affine toric variety} associated with $\sigma$ is 
\[ V_{\sigma} = \text{Spec}(\C[S_{\sigma}]) \]
where $\C[S_{\sigma}]$ is the semigroup algebra associated with $S_\sigma$.

Given a polyhedral fan $\Sigma$ we have a collection of affine toric varieties indexed by cones in $\Sigma$, denoted $\{V_\sigma \ : \ \sigma \in \Sigma \}$.
Specifically, for $\sigma, \tau \in \Sigma$ with $\sigma \subseteq \tau$ we have the inclusion $V_\sigma \subseteq V_\tau$. We denote $X_\Sigma$ as the toric variety that is the colimit of all such inclusion maps.
 This means that $X_\Sigma$ is obtained by gluing all affine varieties $V_\tau$, $V_{\tau'}$ along $V_{\sigma}$, where $\sigma = \tau \cap \tau'$.
For a more complete treatment of toric varieties, see \cite{cox2011toric}.

In this paper, we will work with the \emph{Cox ring} of a toric variety which is a generalization of the homogeneous coordinate ring of projective space and was introduced in \cite{cox1995homogeneous}. First, let us denote by $\Sigma(1)$ the set of all rays of $\Sigma$. By abuse of notation we often do not distinguish between rays $\rho$ and their primitive ray generators.
The Cox ring of $X_{\Sigma}$ is
\[ 
S = \C[x_\rho \ : \ \rho \in \Sigma(1)]. 
\]
To every ray $\rho$ we associate the corresponding torus invariant Weyl divisor
$D_\rho$. Note that 
every torus invariant Weyl divisor $D$ on $X$ has a unique presentation of the form
$D = \sum_{\rho \in \Sigma(1)} a_\rho D_\rho$.
The global sections of the associated sheaf $\mathcal{O}_{X} (D)$ are spanned by monomials:
$$ H^0(\mathcal{O}_{X} (D),y)  = \left\langle y^m \mid  \ m \in \Z^n, \  \langle m, \rho \rangle \geq - a_\rho   \right \rangle \subseteq \C[y_1^\pm, \dots, y_n^\pm] . $$
Given a global section $f = \sum_{m \in \mathbb{Z}^n} c_m y^m$ of $\OO_{X}(D)$ we define the homogenization $\widetilde{f}$ of $f$ to be the following element of $S$:
\begin{equation}
    \label{eq: homogenization}
    \widetilde{f} =\prod_{\rho \in \Sigma(1)} x_\rho^{a_\rho} f(z_1, \dots, z_n) = 
    \sum_{m \in \Z^n} c_m \prod_{\rho \in \Sigma(1)} x_\rho^{ \langle m, \rho \rangle + a_\rho }.
\end{equation}
Here the variables $z_i$ are defined by $z_i = \prod x_\rho^{\rho_i}.$
Note that a Laurent polynomial $f \in \C[y_1^{\pm}, \dots, y_n^{\pm}]$ can be a section of 
the sheaf $\OO_{X}(D)$ for
a certain choices of $D$, and the homogenization $\widetilde{f}$ depends on the choice of $D$.

\begin{example}
Consider the toric variety $\P^1 \times \P^1$, associated with the dual fan $\Sigma(P)$ of the square $P$.
The four generators $x_{-e_1},x_{e_1},x_{-e_2},x_{e_2}$ 
of the Cox ring $S$ are in bijection to the four rays.
\begin{center}
\includegraphics[scale=0.4]{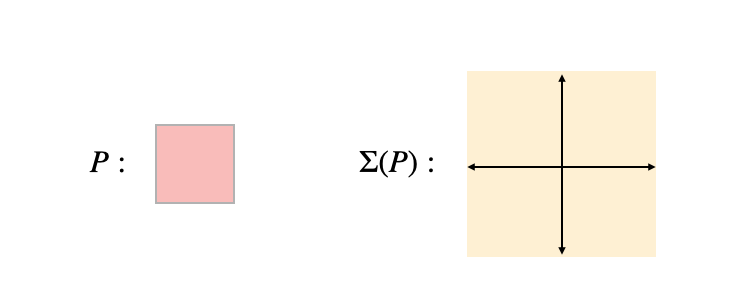}
\end{center}
Homogenizing the bivariate polynomial $f = 1 + y_1 + y_2 + y_1y_2$
yields the bihomogeneous polynomial $\widetilde{f} = x_{-e_1}x_{-e_2} + x_{e_1}x_{-e_2} + x_{-e_1}x_{e_2} + x_{e_1} x_{e_2}$.
\end{example}

\subsection{Chern and Segre classes of vector bundles}
The main ingredient of the intersection theoretic formulas for the algebraic degree of polynomial optimization problems given in \cite{ MR2507133} is Porteous' formula. Porteous' formula computes the expected cohomology class of the degeneracy locus of maps of vector bundles. Loosely speaking, vector bundles are families of vector spaces that are parameterized by another space, and cohomology classes are algebraic invariants of topological spaces. In this paper, all vector spaces will be parameterized by algebraic varieties and the vector spaces will all have the same dimension, called the \emph{rank} of the vector bundle.

To formulate Porteous' formula one needs to use the Chern and Segre classes, which are well-studied characteristic classes of vector bundles. Here we list some main properties of these classes. For a more detailed introduction we refer to \cite{eisenbud_harris_2016}.

For a vector bundle $\EE$ of rank $r$ on a variety $X$ of dimension $d$ and any $i = 0,\ldots,d$, one can define its $i$-th Chern class as an element of the $2i$-th cohomology class of $X$, $c_i(\EE)\in H^{2i}(X)$. For any vector bundle $\EE$, one has $c_0(\EE)=1$ and $c_i(\EE) = 0$ for any $i>r$. We will denote by $c(\EE)$ the \emph{total Chern class} of $\EE$, that is 
\[
c(\EE) = c_0(\EE) +\ldots + c_{{\rm max}(d,r)}(\EE).
\]
The crucial property of total Chern classes, known as Whitney's formula, is that they are multiplicative with respect to taking direct sums of line bundles:
\[
c(\EE\oplus F) = c(\EE)\cdot c(\FF).
\]
In what follows we will mostly work with vector bundles coming as a direct sum of line bundles $E= \LL_1\oplus \ldots \oplus \LL_n$. By applying the Whitney formula to such vector bundles we obtain a convenient formula for their total Chern class:
\[
c(\EE) = \prod_{i=1}^n (1+c_1(\LL_i)), \text{ These are graded pieces: } c_{k}(\EE)= \sum_{I\in \binom{[n]}{k}} \prod_{i\in I} c_1(\LL_i).
\]

Finally, let us recall the definition of Segre classes. Note that, the total Chern class $c(\EE)$ is an invertible element in the cohomology ring of $X$ as its $0$-th degree part is equal to $1$. Using this one defines a total Segre class $s(\EE)$ of a vector bundle $\EE$ on $X$ to be the inverse of the total Chern class of $\EE$:
\[
s(\EE)= c(\EE)^{-1}. 
\]
Individual Segre classes $s_i(\EE)$ are defined as homogeneous components of the total Segre class. Note that, unlike Chern classes, one could have non-trivial Segre class $s_i(\EE)$ even for $i>r$. 

\section{Statement of the main result}
In this section, we give an overview of the main results of this paper and defer the proofs of \Cref{thm: A} and \Cref{thm: C} to \Cref{sec: secion_with_proofs_for_thm_A_and_B}.
The results in this paper require certain assumptions on the monomial support of $\mathbf{F}$, which we define in this section.
While computational experiments suggest that both \Cref{thm: A} and \Cref{thm: B} are true under milder assumptions, we show in \Cref{example: demonstrate assumptions for main theorem}
that they are needed for \Cref{thm: C}.

\begin{definition}
\label{def: admissible point configuration}
    We call a point configuration $\AA = (\AA_0,\ldots,\AA_m) \in \N^n$ \emph{admissible} if for every $i = 0,\ldots, m$:
    \begin{enumerate}
        \item $\R_{\geq 0 }^n$ is a cone in the normal fan of the polytope $\conv(\AA_i)$, and
        \item  $\conv(\AA_i)$
    meets every coordinate hyperplane of $\R^n$.
    \end{enumerate}
\end{definition}

When passing to a toric compactification, for various technical reasons we need to ensure
that the constraints $f_1, \dots, f_m$ define a variety with a smooth closure.
This is guaranteed by the following notion of an \emph{appropriate} toric variety.

\begin{definition}
\label{def: appropriate}
    Let $X$ be a proper, normal toric variety with underlying polyhedral fan $\Sigma$ in $\R^n$
    and $\AA = (\AA_0,\ldots,\AA_m)$ an admissible point configuration.
    We call $X$ \emph{appropriate} for $\AA$ if the following three properties hold:
    \begin{enumerate}
        \item the normal fan $\Sigma(\conv(\AA_i))$ is refined by $\Sigma$ for each $i = 0, \dots, m$,
        \item the fan $\Sigma$ contains the non-negative orthant $\R^n_{\geq 0}$, and
        \item for generic functions $f_i$ with monomial support $\AA_i$ the closure of $\VV(f_1, \dots, f_m)$ in $X$ 
        is disjoint from the singular locus of $X$.
    \end{enumerate}
\end{definition}

\begin{remark}
    \label{existence of admissible X}
    Note that for every admissible point configuration $\AA$, there exists a smooth appropriate toric variety $X$.
    To construct $X$ consider the normal fan $\Sigma'$ of the Minkowski sum
    $ \conv(\AA_0) + \cdots + \conv(\AA_m)$.
    A resolution of singularities can be performed by subdividing each singular cone of $\Sigma'$, resulting in a smooth,
    complete polyhedral fan $\Sigma$ which contains the positive orthant.
    For more details on toric resolution of singularities consider Chapter 11 of \cite{cox2011toric}.
\end{remark}

In the proof of \Cref{thm: A} and \Cref{thm: C} we consider a natural choice for the toric compactification $X$, given by the coarsest refinement of all normal fans of
the Newton polytopes $\newt(f_0), \dots, \newt(f_m)$.
To state our main results, we need one final definition.

\begin{definition}
\label{def: strongly admissible point configuration}
    We call a point configuration $\AA = (\AA_0,\ldots,\AA_m) \in \N^n$ \emph{strongly admissible} if
    it is admissible and if the toric variety
    \[X = X(\Sigma(\conv(\AA_0) + \cdots + \conv(\AA_m)))\]
    is appropriate for $\AA$.
    Here $X$ is the toric variety
    associated with the common refinement of the normal fans $\Sigma(\conv(\AA_i))$.
\end{definition}
We give examples for the above definitions.

\begin{example}
\label{ex: examples for admissibility, appropriate, etc}
Consider the tetrahedron $\mathcal{S} = \conv(0, e_1, e_2, e_3) \subset \R^3$, 
the cube $\mathcal{C} = \conv(0, e_1, e_2, e_3, e_1+e_2, e_1+e_3, e_2+e_3, e_1+e_2+e_3) \subset \R^3$
and the bipyramid $\mathcal{B} = \conv(0, e_1, e_2, e_3, e_1+e_2+e_3) \subset \R^3$.
    \begin{itemize}
    \item The point configuration $(\mathcal{S}, - \mathcal{S} + (1,1,1) )$ is not an admissible point configuration since the normal fan of $- \mathcal{S}$ does not contain $\mathbb{R}^3_{\geq0}$ as a cone.
        \item The point configuration $(\mathcal{S}, \mathcal{C})$ is a strongly admissible point configuration since it is admissible and the singular locus of the toric variety defined by $\mathcal{S}+\mathcal{C}$ is zero-dimensional.
       \item The point configuration $(\mathcal{S}, \mathcal{B})$ is an admissible, but not strongly admissible, point configuration. In particular, the toric variety defined by $\mathcal S+\mathcal B$ is not appropriate for $(\mathcal{S},\mathcal{B})$. To see this consider the three-dimensional toric variety $X$ defined by $\mathcal{S}+\mathcal{B}$. The torus orbit 
 $X_\sigma$ defined by the cone $\sigma =  \R_+ (-1,-1,1)+\R_+ (-1,-1,-1) $, dual to the face $\conv((2,1,1), (1,2,1))$ of $\mathcal{S}+\mathcal{B}$, is contained in the singular locus.
       Let now $f_1$ be a generic polynomial with Newton polytope $\mathcal{S}$. Then the vanishing locus $\VV(f)$ in $X$ intersects $X_\sigma$.

    \end{itemize}
\end{example}
The polytopes $\mathcal{S}+\mathcal{B}$ and $\mathcal{S}+\mathcal{C}$ are displayed below, with the faces that define singular torus orbits colored in green.
    \begin{center}
\includegraphics[scale=0.2]{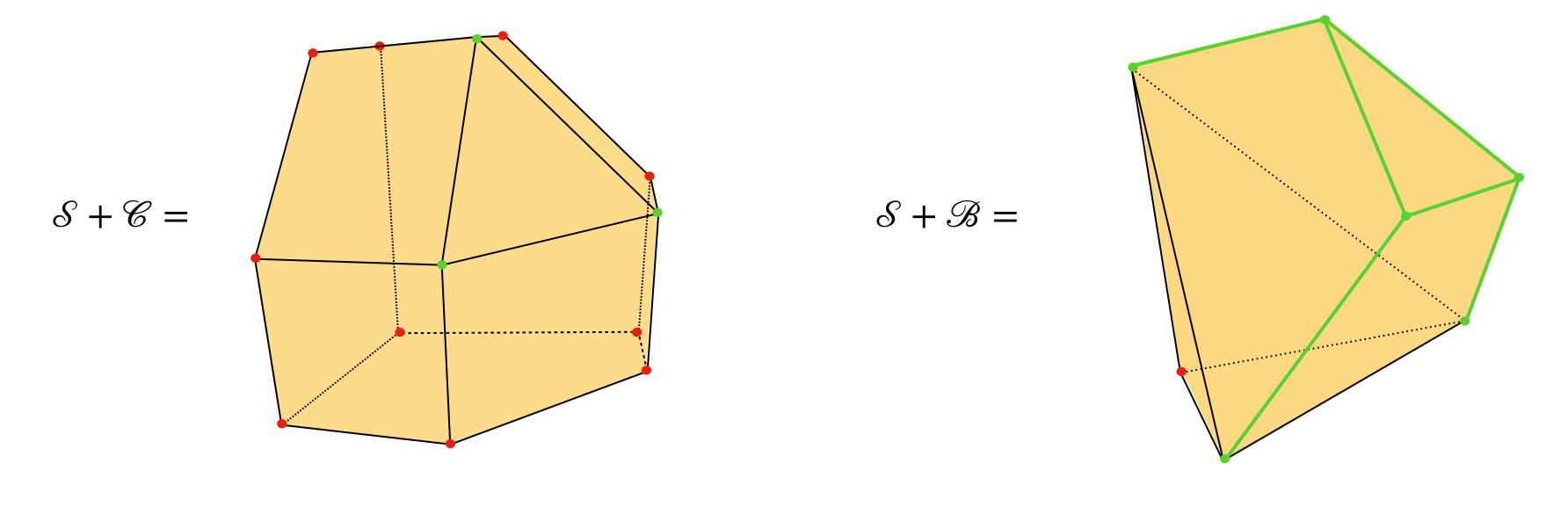}
\end{center}

The next two results express the number of solutions of $\mathbf{L}_{\mathrm{F}}$ as mixed volumes.

\begin{thmx}
\label{thm: A}
Let $\mathbf{F} =(f_0,\ldots,f_m)$ be a generic sparse system of polynomials 
in $\mathbb{C}[x_1,\ldots,x_n]$
with strongly admissible support. Then
the 
algebraic degree of sparse polynomial optimization of
\eqref{POP}
is equal to 
\begin{align}
\label{eq:Mixed_Volume_partialDelta}
\operatorname{MV} \left( \newt(f_1), \dots, \newt(f_m),
\partial_1  \newt( \Phi_F ) ,\dots,
\partial_n  \newt(\Phi_F )
\right).
\end{align}
\end{thmx}

Here $\Phi_F$ denotes the Lagrangian $\Phi_F(\lambda, y) := f_0 - \sum_{i=1}^m \lambda_i f_i$, as above.
Note that the polytopes $\partial_j  \newt(\Phi_F)$ may be strictly larger than the Newton polytopes
$\newt (\ell_j )$ of the partial derivatives $\ell_j = \frac{\partial}{\partial y_j} \left( f_0 - \sum_{i=1}^m \lambda_i f_i \right)$. The equality in \Cref{thm: A} shows that $\mathbf{L}_{\mathrm{F}}$ is BKK general, giving the next theorem.

\begin{thmx}
\label{thm: B}
Under the assumptions of \Cref{thm: A}
the Lagrange system $\mathbf{L}_{\mathrm{F}}$
is BKK general. 
Further, all critical points are smooth and lie in the algebraic torus $(\C^*)^n$. The number of critical points is equal to the mixed volume
\begin{equation}
\label{eq:Mixed_Volume_Lagrangian}
\operatorname{MV} \left( \newt (f_1), \dots, \newt (f_m),
\newt (\ell_1 ),\dots,
\newt ( \ell_n)
\right).
\end{equation}
If $\mathbf{F}$ is not generic then \eqref{eq:Mixed_Volume_Lagrangian} is an upper bound to the number of isolated, smooth critical points of $f_0$ restricted to $\VV(f_1, \dots, f_m)$.
\end{thmx}
\begin{proof}
For every $j = 1, \dots, n$ we have the inclusion
$\newt (\ell_j ) \subseteq \partial_j  \newt(\Phi_F)$
of polytopes, showing the inequality
\eqref{eq:Mixed_Volume_Lagrangian}$\leq$\eqref{eq:Mixed_Volume_partialDelta}.
On the other hand, by \Cref{thm: A}, the BKK
bound \eqref{eq:Mixed_Volume_Lagrangian}
of $ \mathbf{L}_{\mathrm{F}} $ constitutes an upper bound to \eqref{eq:degree of sparse pol optimization} and we obtain
\[
\eqref{eq:degree of sparse pol optimization} \leq 
\eqref{eq:Mixed_Volume_Lagrangian} \leq
\eqref{eq:Mixed_Volume_partialDelta}=
 \eqref{eq:degree of sparse pol optimization}.
\]
By \Cref{lemma: transversality upstairs}, all critical points are smooth and all critical points are in the torus. 
\end{proof}

By \Cref{thm: A} the monomial supports $\AA_0, \dots, \AA_m$ contribute to the mixed volume \eqref{eq:Mixed_Volume_partialDelta} only by means of their convex hulls. This gives the following corollary.

\begin{cor}\label{cor: delete mons}
    Under the assumptions of \Cref{thm: A} the algebraic degree of the sparse polynomial optimization problem \eqref{POP}
    and the mixed volume \eqref{eq:Mixed_Volume_Lagrangian}
    depend only on the Newton polytopes $\newt(f_i)$ for $i = 0, \dots, m$.
\end{cor}

\begin{remark}
    \Cref{cor: delete mons} has algorithmic consequences if one wishes to numerically find all critical points to \eqref{POP} (as opposed to counting them). We leverage this in \cite{rose2023polyhedral} and efficiently compute polyhedral start systems for $\mathbf{L}_{\mathrm{F}}$ by imposing maximal sparsity on the Newton polytopes of $\mathbf{F}$.
\end{remark}

\begin{example}
Recall the optimization problem \eqref{opt ex} in \Cref{ex1}. \Cref{thm: B} shows that the algebraic degree of this problem is equal to the mixed volume of its Lagrange system. 
A property of mixed volumes is that if $P_1,\ldots,P_n \subset \R^n$ and $Q_1,\ldots,Q_m \subset \R^{n+m}$, then 
\[\mvol(P_1,\ldots,P_n,Q_1,\ldots,Q_m) = \mvol(P_1,\ldots,P_n) \cdot \mvol(\pi(Q_1),\ldots,\pi(Q_m)),\]
where $\pi:\R^{n+m} \to \R^m$ is the projection onto the last $m$ coordinates. Observe that $\newt(\ell_1),\ldots,\newt(\ell_n) \subset \R^{n+1}$ have $n$th coordinate zero. Therefore, 
\small \[ 
\mvol(\newt(\ell_1),\ldots,\newt(\ell_n),\newt(f)) = \mvol(\newt(\ell_1),\ldots,\newt(\ell_n)) \cdot \mvol (\pi_n(\newt(f))
 \]
 \normalsize
where $\pi_n: \R^{n+1} \to \R $ is the projection onto the $n$th coordinate. 

Since for $j = 1,\ldots, n$, $\newt(\ell_j) = \conv(0, \alpha_j)$ for some $\alpha_j \in \R^n$, we can compute $\mvol(\newt(\ell_1),\ldots,\newt(\ell_n)) = \det(M)$ where $M$ is the matrix with $j$th column equal to $\alpha_j$. In our case, this amounts to computing
\begin{align*}
    \det\left( \begin{bmatrix}
        2 & 0  &\ldots & 0\\
        0 & 1   & \ldots & 0 \\
        0 & 0  &  \ddots   & 0 \\
        1 & 1    & \ldots & 1
    \end{bmatrix} \right) = 2.
\end{align*}

Finally, observe that $\pi_n(\newt(f)) = [0,1]$ so it has (mixed) volume one. This gives a geometric proof that the optimization degree of \eqref{opt ex} is $2$, agreeing with the result we computed in \Cref{ex1}.

\end{example}

Our final result, \Cref{thm: C}, weakens the assumption of strongly admissible support needed in \Cref{thm: A} and \Cref{thm: B} and characterizes the algebraic degree of $ \mathbf{F} $ as a more general product in the Chow ring of a toric variety.

To formulate \Cref{thm: C} we first need some notation. We refer to \Cref{sec: prelims} and references therein for a brief introduction to the objects we use. Let  $\AA= (\AA_0,\ldots,\AA_m)$ be an admissible point configuration (\Cref{def: admissible point configuration}) and let $X$ be a smooth toric variety given by the fan $\Sigma$ which is appropriate for $\AA$ (\Cref{def: appropriate}). As usual, the convex hull of each point configuration $\AA_i$ defines a line bundle $\LL_{\AA_i}$. 

Further, since we assume that the fan $\Sigma$ contains the positive orthant as one of its cones, we know that $\Sigma$ contains the rays generated by the standard basis vectors $e_1,\ldots,e_n$ of $\R^n$. We denote by $D_{e_1},\ldots, D_{e_n}$ the corresponding torus-invariant divisors on $X$ and by $\OO_X(D_{e_1}),\ldots,\OO_X(D_{e_n})$ the corresponding line bundles. 
For further reading on toric line bundles, see \cite[Chapter 6]{cox2011toric}.

\begin{thmx}
\label{thm: C}
Let $\mathbf{F} =(f_0,\ldots,f_m)$ be a generic sparse system of polynomials 
in $\mathbb{C}[x_1,\ldots,x_n]$
with admissible support $\AA$
and let $X$ be as above.
Then the algebraic degree of sparse polynomial optimization
\eqref{POP}
is finite and equal
to the degree of the following cycle class:
\begin{equation}
\label{eq: numberofcriticalpoints_via_cohomology}
    c_{1} \left( \LL_{\AA_1}   \right) 
 \cdots 
 c_{1} \left( \LL_{\AA_n}\right)
\left( \operatorname{s}( \EE ) \chern(\FF ) \right)_{n-m},
\end{equation}
where $\EE = \LL_{\AA_0}^{-1}\oplus \cdots \oplus \LL_{\AA_m}^{-1}$ and   $\FF = \OO_X(-D_{e_1}) \oplus \cdots \oplus \OO_X(-D_{e_n}).$

Moreover, if $\mathbf{F}$ is not generic then \eqref{eq: numberofcriticalpoints_via_cohomology} is an upper bound to the number of isolated, smooth critical points of $f_0$ restricted to $\VV(f_1, \dots, f_m)$.
\end{thmx}

The purpose of the following example is to demonstrate that
the assumptions
for \Cref{thm: C}
 are necessary.

\begin{example}
\label{example: demonstrate assumptions for main theorem}

Consider \eqref{POP} where
\begin{align*}
f_0 &= 7 + 11y_1 - 13y_2 - 19y_1y_2 - 2y_1^2 - 5y_2^2 \\
f_1 &= -5y_1y_2 + 29y_1y_2^2 - 17y_1^2y_2 + 61y_1^2y_2^2 + y_1^2 - 3y_2^2.
\end{align*}

Note that the normal fan of $\newt(f_1)$ does not contain $\R^2_{\geq 0}$, so the monomial support of $(f_0,f_1)$ does not form an admissible point configuration, meaning the assumptions of \Cref{thm: C} do not hold in this case.

\begin{center}
      \begin{tikzpicture} [scale = 0.75]
   
   \node at (-1.5,0){ 
 \begin{tikzpicture}[scale = 2]
 
            \coordinate (node) at (-0.6, 0.5);
            \coordinate (O) at (0,0);
            \coordinate (X1) at (1, 0) ;
            \coordinate (X2) at (0, 1) ;
            \node at (node) {$\newt(f_0) :$};
                \node at (1.3,0) {};
                \node at (0,1.2) {};
            \draw[fill=green,opacity=0.3]  (O) -- (X2) -- (X1) -- (O);
 \end{tikzpicture}
 };

   \node at (4,0){ 
 \begin{tikzpicture}[scale = 2]
 
            \coordinate (node) at (-0.5, 0.5);
            \coordinate (O) at (0,0);
            \coordinate (X1) at (1, 0) ;
            \coordinate (X2) at (0, 1) ;
                        \coordinate (X3) at (1,1);

            \node at (node) {$\newt(f_1) :$};

                \node at (0.65,0) {};
                \node at (0,1.2) {};
                \node at (1,1.2) {};

             \draw[fill=green,opacity=0.3]  (X3) -- (X2) -- (X1) -- (X3);

 \end{tikzpicture}
 };
        \end{tikzpicture}  
\end{center}

\noindent To evaluate \eqref{eq: numberofcriticalpoints_via_cohomology}
we denote the Chern classes
\[
     c_{1} \left( \LL_{\AA_0}\right) = [2D_{2} + 4D_{3} + 2D_{4} -2 D_{6}]  \text{ and }     c_{1} \left( \LL_{\AA_1}\right)
      =  [2D_{2} + 2D_{3} + 2D_{4}], 
\]
and the vector bundles
\[
\EE = \LL_{\AA_0}^{-1} \oplus  \LL_{\AA_1}^{-1}   \text{ and } \FF = \OO_X(D_1)^{-1} \oplus \OO_X(D_5)^{-1}.
\]
Here, $X$ is the smooth toric variety defined by the complete fan $\Sigma$
with ray generators
\[
\rho_1 = (0,1), \ \rho_2 = (1,1), \ \rho_3 = (1,0), \ \rho_4 = (0,-1), \ \rho_5 = (-1,-1), \ \rho_6 = (-1,0).
\]
\begin{center}
        \begin{tikzpicture} 
            \coordinate (node) at (0, 0);
            \coordinate (O) at (0,0);
            \coordinate (X0) at (0, 1) ;

            \coordinate (X1) at (1, 1) ;
            \coordinate (X2) at (1, 0) ;
            \coordinate (X3) at (0, -1) ;
            \coordinate (X4) at (-1, -1) ;
            \coordinate (X5) at (-1, 0) ;
    
            \filldraw [black] (O) circle (2pt); 
            \draw[fill=yellow,opacity=0.3]  (O) -- (X0) -- (X1);
            \draw[fill=yellow,opacity=0.3]  (O) -- (X2) -- (X1);
            \draw[fill=yellow,opacity=0.3]  (O) -- (X3) -- (X2);
            \draw[fill=yellow,opacity=0.3]  (O) -- (X4) -- (X3);
            \draw[fill=yellow,opacity=0.3]  (O) -- (X4) -- (X5);
            \draw[fill=yellow,opacity=0.3]  (O) -- (X5) -- (X0);

            \draw[draw = black, very thick] (O) --(X0);
            \draw[draw = black, very thick] (O) --(X1);
            \draw[draw = black, very thick] (O) --(X2);
            \draw[draw = black, very thick] (O) --(X3);
            \draw[draw = black, very thick] (O) --(X4);
            \draw[draw = black, very thick] (O) --(X5);

    \node at (0,1.5) {$\rho_1$};
    \node at (1.3,1.3) {$\rho_6$};
    \node at (1.5,0) {$\rho_5$};
    \node at (-1.3,-1.3) {$\rho_3$};
                \node at (0,-1.5) {$\rho_4$};
                \node at (-1.5,0) {$\rho_2$};
        \end{tikzpicture}
\end{center}

We have $c_1(\FF) = [-D_{1} - D_{5}]$, and 
\[s_1( \EE )  =  \frac{-1}{ c_1(\EE)   } = c_{1} \left( \LL_{\AA_0}\right) +c_{1} \left( \LL_{\AA_1}\right) = [4D_{2} + 6D_{3} + 4D_{4}  -2 D_{6}].
\]
Finally, direct computation shows 
\[
\operatorname{deg} \left(
c_{1} \left( \LL_{\AA_1}   \right) \cdot \left( \operatorname{s}( \EE ) \chern(\FF ) \right)_{1} \right)  = 
\operatorname{deg} \left(
c_{1} \left( \LL_{\AA_1}   \right) \cdot  \left(    c_1(\FF) +  s_1( \EE )       \right) \right) = 12.
\]
In this case,
\eqref{eq: numberofcriticalpoints_via_cohomology} gives $12$, 
while the algebraic degree of $\mathbf{F} = (f_0,f_1)$ is $10$.
In particular, the this shows the assumptions of \Cref{thm: C} are necessary. On the other hand, the BKK bound of $\mathbf{L}_{\mathrm{F}}$ is $10$, so the bound in
\Cref{thm: A} and \Cref{thm: B} is tight.
Although the assumptions in \Cref{thm: A} and \Cref{thm: B} are stronger than the ones in \Cref{thm: C}, we believe that \Cref{thm: A} and \Cref{thm: B} hold in greater generality than \Cref{thm: C}.

\end{example}

\section{Sparse ED, polar and sectional degrees}
In this section, we discuss important corollaries of \Cref{thm: A} and \Cref{thm: B} which relate Euclidean distance optimization, polar degrees, and sectional degrees to mixed volumes.

We first consider polynomial optimization problems where the objective function is $f_0 = \lVert y - u \rVert_2^2$  for a generic point $u \in \R^n$. Let $(f_1,\ldots,f_m )$ be a sparse polynomial system with variety $Y = \VV(f_1,\ldots,f_m) \subset \C^n$. The \emph{ED degree} of $Y$ is the number of complex critical points of the optimization problem:
\begin{align}
    \min_{y \in \R^n} \ \lVert y - u \rVert_2^2 \quad \text{subject to} \quad f_1(y) = \ldots = f_m(y) = 0. \tag{ED} \label{ED}
\end{align}
Equivalently, it is the algebraic degree of $\mathbf{F} = (f_0,\ldots,f_m)$.
This brings us to the main result of this section, which relates ED degrees and mixed volumes.
To simplify notation, for polynomials $\mathbf{F} =(f_1,\ldots,f_m)$ we write
\[\mvol(F) := \mvol(\newt(f_1),\ldots,\newt(f_m)).\]

\begin{cor}[Euclidean distance objective function]\label{cor:EDdegree}
Let $(f_1,\ldots,f_m)$ be a generic sparse polynomial system with variety $\VV(f_1,\ldots,f_m) = Y \subset \C^n$ and $f_0 = \lVert y -  u \rVert_2^2$ where $u$ is a generic point in $\R^n$.
If the monomial support of $\mathbf{F} = (f_0,\ldots,f_m )$ is strongly admissible, then 
\[ \mvol(\mathbf{L}_F) = \text{ED degree}(Y). \]
\end{cor}
\begin{proof}
Consider the \emph{weighted Euclidean distance function}, $f_C = \lVert Cy - u \rVert_2^2$, where $C$ is an $n \times n$ diagonal matrix with generic entries and define $\mathbf{F}_C = (f_C, f_1,\ldots,f_m )$. 
\Cref{thm: A} implies that the degree and mixed volume of $\mathbf{L}_{\mathrm{F}_C}$ are equal. Call this value $\eta$.

Observe that the variety of $\mathbf{L}_{\mathrm{F}_c}$ is in bijection with the critical points of

\begin{align} 
\min_{y \in \mathbb{R}^n} \  \lVert y - u \rVert_2^2 \quad \text{subject to} \quad f_1(C^{-1}y) = \cdots = f_m(C^{-1}y) =  0. \label{eq:ED}
\end{align}

This gives that there are $\eta$ critical points to \eqref{eq:ED}. 
Notice that the monomial support of $(f_1(C^{-1}(y)),\ldots,f_m(C^{-1}(y))$ is the same as that of $(f_1(y),\ldots,f_m(y))$, since $C$ is diagonal. 
Therefore, the degree of $\mathbf{L}_{F}$ is equal to its mixed volume.
\end{proof}

Now that we have established a relationship between ED degrees and mixed volumes, we next recall that in \cite{EDdegree2016} a relationship between ED degrees and polar degrees was established. 
Let $\overline{Y} \subset \P^{n}$ be a projective variety.
For a point $y \in \overline{Y}$, denote $T_x\overline{Y}$ as the tangent space of $\overline{Y}$ at $y$. Denote the conormal variety of $\overline{Y}$ as 
\begin{align}
\mathcal{N}_{\overline{Y}} = \overline{\{ (y,z) \in \P^{n} \times (\P^{n})^* \ : \ y \in \overline{Y}_\text{sm}, \ z \perp T_y\overline{Y} \} }, \label{eq: conormal}
\end{align}
where $(\mathbb{P}^n)^*$ is the \emph{dual projective space}, which is the projective space consisting of all hyperplanes in $\mathbb{P}^n$. For a proper subvariety, $\overline{Y} \subset \mathbb{P}^n$, the dimension of $\mathcal{N}_{\overline{Y}}$ is independent of the dimension of $\overline{Y}$ and $\dim (\mathcal{N}_{\overline{Y}}) = n-1$.

    For an irreducible variety $\overline{Y} \subset \P^{n}$, the \emph{$i$-th polar degree of $\overline{Y}$} is 
    $$\delta_i(\overline{Y}) = | \mathcal{N}_{\overline{Y}} \cap (L_1 \times L_2)|,$$
    where $L_1 \subseteq 
    \P^{n}$
    is a generic linear space of dimension
    $n + 1 - i$ and $L_2 \subseteq (\P^{n})^*$ is a 
    generic linear space of dimension $i$.
    The intersection 
    $\mathcal{N}_X \cap (L_1 \times L_2)$ is finite, since
    $$\dim(\mathcal{N}_{\overline{Y}}) + \dim(L_1 \times L_2) = 2n = \dim(\P^{n} \times (\P^{n})^*).$$

While in our situation, the requirement that $\overline{Y}$ is an irreducible projective variety 
does not typically hold, we remark that by considering an affine variety $Y \subset \C^n$ defined by polynomial equations with strongly admissible support, we can simply consider the \emph{projective closure} of $Y$, which we denote $\overline{Y}$. 
Under mild genericity conditions, the projective closure of $Y$ is defined by homogenizing the defining equations of $Y$ with respect to a new variable, $y_0$. 

In general the ED degree of $Y$ need not be equal to that of $\overline{Y}$ \cite[Example 6.6]{EDdegree2016}, but 
under certain transversality conditions they are. Given two varieties $V, W \subseteq \mathbb{P}^n$, we say that the the intersection $V \cap W$ is \emph{transversal} if the scheme theoretic intersection $V \cap W$ is smooth.
Let $H_{\infty} = \P^n \backslash \C^n = \VV(y_0)$ denote the hyperplane at infinity and denote $Y_{\infty} = \overline{Y} \cap H_{\infty}$ and $Q_{\infty} = \{y_0^2 + \ldots + y_n^2 = 0\}$.

\begin{theorem}[Theorem 6.11 \cite{EDdegree2016}]\label{thm:affine ed polar}
Let $Y \subset \C^n$ be an irreducible, affine variety 
and $\overline{Y} \subset \P^n$ its projective closure. Assume that the intersections $ \overline{Y} \cap H_{\infty}$ and $Y_{\infty} \cap Q_{\infty}$ are both transversal. Then,
 \[\text{ED Degree}(Y) = \delta_0(\overline{Y}) + \delta_1(\overline{Y}) + \cdots + \delta_{n-1}(\overline{Y})\]
where $\delta_i(\overline{Y})$ is the $i$-th polar degree of $\overline{Y}$.
\end{theorem}

As a consequence of \Cref{cor:EDdegree} and \Cref{thm:affine ed polar} we establish a relationship between polar degrees and mixed volumes. To our knowledge, this is the first time a connection between convex geometry and polar degrees has been made.

\begin{cor}\label{cor:mvol and polar}
    Let $Y \subset \C^n$ be an irreducible affine variety defined by polynomials $(f_1,\ldots,f_m)$ 
    and let $\overline{Y} \subset \P^n$ be its projective closure. Assume that the intersections $Y_{\infty} = \overline{Y} \cap H_{\infty}$ and $Y_{\infty} \cap Q_{\infty}$ are both transversal. Let $\mathbf{L}_{\mathrm{F}}$ be the Lagrange system of $\mathbf{F} = (f_0,\ldots,f_m)$ where $f_0 = \lVert y - u \rVert_2^2$ for generic $u \in \mathbb{R}^n$ and $\mathbf{F}$ has strongly admissible support. Then
    \begin{align*}
        \mvol(\mathbf{L}_{\mathrm{F}}) = \delta_0(\overline{Y}) + \cdots + \delta_{n-1}(\overline{Y}),
    \end{align*}
    where $\delta_i(\overline{Y})$ is the $i$-th polar degree of the projective closure $\overline{Y}$.
\end{cor}

\begin{example}
    Consider the Euclidean distance optimization problem
    \begin{align*}
        \min_{y \in \R^2} \ \left\lVert \begin{bmatrix}
            y_1 \\ y_2
        \end{bmatrix} -\begin{bmatrix}
            1 \\ 1
        \end{bmatrix} \right\rVert_2^2 \quad \text{subject to} \quad 4y_1^2 +2y_2^2 - y_1y_2 -1 = 0 .
    \end{align*}
    The Lagrange system of this optimization problem is
    \begin{align*}
        \mathbf{L}_{\mathrm{F}} =  (2 (y_1 - 1) - \lambda (8 y_1 - y_2), 2 (y_2 - 1) - \lambda (-y_1 + 4 y_2), 4y_1^2 +2y_2^2 - y_1y_2 - 1  ).
    \end{align*}
    The mixed volume of $\mathbf{L}_{\mathrm{F}}$ is four and there are indeed four complex solutions $(y_1,y_2,\lambda)$, two of which are real:
    \begin{align*}
        &[0.3864, 0.5557 , -0.4840 ], \\
        &[-0.3050, -0.6418 , 1.4516],\\ 
        &[-0.1407 - 1.6318 \textbf{i}, 2.3431 - 0.5950 \textbf{i}, 0.2904 - 0.1023 \textbf{i}],\\
        &[-0.1407 + 1.6318 \textbf{i}, 2.3431 + 0.5950 \textbf{i}, 0.2904 + 0.1023\textbf{i}]
    \end{align*}
\begin{figure}[h!]
    \centering
    \includegraphics[width = 0.5\textwidth]{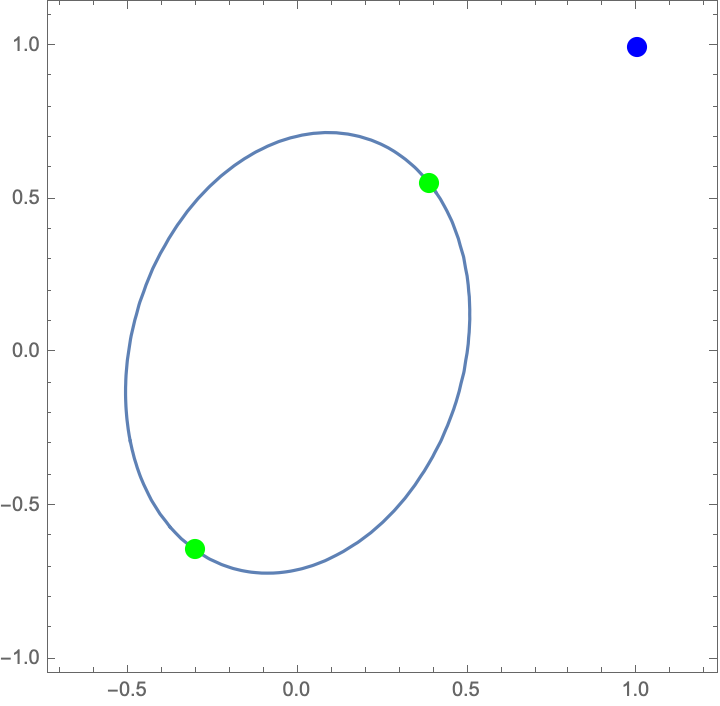}
    \caption{The ellipse $4y_1^2 + 2y_2^2 - y_1y_2 -1 = 0$ along with the critical points (green) of the Euclidean distance problem from the point $(1,1)$ (blue).}
    \label{fig:euclid ex}
\end{figure}

Now we consider the projective closure of the variety 
\[Y = \mathcal{V}(4 y_1^2 + 2y_2^2 - y_1y_2 - 1) \subset \C^2 \]
which is defined as 
\[\overline{Y} = \VV(4y_1^2 + 2y_2^2 - y_1 y_2 - y_0^2) \subset \P^2.\] 
The conormal variety of $\overline{Y}$, $\mathcal{N}_{\overline{Y}} \subset \P^2 \times (\P^2)^*$, is defined as
\begin{align*}
    \mathcal{N}_{\overline{Y}} &= \{([y_0 : y_1 : y_2], [z_0:z_1:z_2]) \in \mathbb{P}^2 \times (\P^2)^* \ : f_1 = f_2 = f_3 = f_4 = f_5 = f_6 = 0 \}
\end{align*}
where the polynomials $f_1,f_2,f_3,f_4,f_5,f_6$ are
\begin{align*}
 f_1 &= -y_0^2+4y_1^2-y_1y_2+2y_2^2 \\  
 f_2 &= y_1z_1-4y_2z_1+8y_1z_2- y_2z_2 \\ 
 f_3 &= 31z_0^2-8z_1^2-4z_1z_2-16z_2^2 \\
 f_4 &= 31y_2z_0+2y_0z_1+16y_0z_2 \\
 f_5 &= 31y_1z_0+8y_0z_1+2y_0z_2\\ 
 f_6 &=y_0z_0+4y_2z_1-8y_1z_2+2y_2z_2 .
\end{align*}
Direct computation shows that $\delta_0(\overline{Y}) = 2$ and $\delta_1(\overline{Y}) = 2$, and that the ED degree of $Y$ equals the sum of the polar degrees of its projective closure as expected.

We conclude this section by making a final connection to sectional degrees. The notion of \emph{sectional degrees} was recently studied in \cite{maxim2023linear}. Given an affine variety $Y \subset \C^n$, the \emph{$i$-th sectional degree} of $Y$, denoted $s_i(Y)$, is defined as the algebraic degree of the optimization problem
\begin{align}
    \min_{y \in \R^n} \ \langle u,y \rangle \quad \text{subject to} \quad y \in Y \cap H_1 \cap H_2 \cap \cdots \cap H_i \tag{SO$_i$} \label{opt: SO}
\end{align}
where $\langle u, \cdot \rangle $ is a generic linear function and $H_1,\ldots, H_i$ are generic affine linear hyperplanes. As an immediate consequence of \Cref{thm: B} we have a convex algebraic interpretation of $s_i(Y)$.

\begin{cor}\label{cor: mvol sectional}
    Let $Y \subset \C^n$ be an affine variety defined by generic polynomials $(f_1,\ldots,f_m )$ where for generic $u \in \mathbb{R}^n$, $\mathbf{F} = (\langle u, y \rangle, f_1,\ldots, f_m)$ has strongly admissible support. Let $\mathbf{L}_{\mathrm{F}}$ be the Lagrange system of $\mathbf{F} $ corresponding to the sectional optimization problem \eqref{opt: SO}. Then
    \begin{align*}
\mvol(\mathbf{L}_{\mathrm{F}}) = s_i(Y)
    \end{align*}
    where $s_i(Y)$ is the $i$-th sectional degree of $Y$.
\end{cor}

Furthermore, \cite{maxim2023linear} gives a condition where sectional and polar degrees agree. This condition relies on the \emph{dual variety} of $\overline{Y} \subseteq \mathbb{P}^n$ which is defined as 
$$ \overline{Y}^{\vee} = \pi_2(\mathcal{N}_{\overline{Y}})$$
where $\pi_2: \mathbb{P}^n \times (\mathbb{P}^n)^* \to (\mathbb{P}^n)^*$ is the projection onto the second coordinate and $\mathcal{N}_{\overline{Y}}$ is the conormal variety of $\overline{Y}$ as defined in \eqref{eq: conormal}.

Specifically, \cite[Corollary 6.8]{maxim2023linear} states that if $Y \subset \C^n$ is an affine variety with projective closure $\overline{Y} \subset \P^n$ such that $H_{\infty}$ is not contained in $\overline{Y}^{\vee}$, then $s_i(Y) = \delta_{i}(\overline{Y})$ for all $0 \leq i \leq \dim (Y)$. Therefore, as a corollary of this, \Cref{cor:mvol and polar} and \Cref{cor: mvol sectional} we get the following equality of mixed volumes.

\begin{cor}
    Let $Y \subset \C^n$ be an affine variety defined by polynomials $(f_1,\ldots,f_m)$ such that for generic $u \in \R^n$ the polynomial systems
    \begin{align*}
        \mathbf{E} &= (\lVert y - u \rVert_2^2,f_1\ldots,f_m), \quad \text{and} \\
        \mathbf{F} &= (\langle u,y \rangle, f_1,\ldots, f_m)
    \end{align*}
    have strongly admissible support. Let $\mathbf{L}_{\mathrm{E}}$ be the Lagrange system of $\mathbf{E} $ corresponding to the Euclidean distance optimization problem \eqref{ED} and $\mathbf{L}_{F_i}$ the Lagrange system of $\mathbf{F}_i$ corresponding to the $i$-th sectional optimization problem \eqref{opt: SO}. Assume that $H_{\infty}$ is not contained in the dual variety of $\overline{Y}$. Then 
    \begin{align*}
        \mvol(\mathbf{L}_{\mathrm{E}}) &= \sum_{i=0}^{n-1} \mvol(\mathbf{L}_{\mathrm{F}_i}).
    \end{align*}
\end{cor}

Observe that under certain genericity conditions, the results in \cite{maxim2023linear} show that sectional degrees are the affine analog of polar degrees. With this in mind and the aforementioned results, we have the following conjecture.

\begin{conj}\label{conj: sectional}
    Let $Y \subseteq \C^n$ be an irreducible, affine variety and $\overline{Y} \subset \P^n$ its projective closure. If $\overline{Y}$ intersects $Q_{\infty}$ transversely
    then
    the ED degree of $Y$ is equal to $s_0(Y) + \ldots + s_{n-1}(Y)$.
\end{conj}

To provide one piece of evidence for \Cref{conj: sectional} and highlight its distinction from \Cref{thm:affine ed polar}, we give an example of a variety $Y$ where \Cref{conj: sectional} is true but \Cref{thm:affine ed polar} gives a strict upper bound on the ED degree of $Y$.

\begin{example}
    Consider the affine variety $Y = \VV(y_1^2 - y_2) \subset \R^2$. We can directly compute the ED degree of $Y$ to be three. The sectional degrees of $Y$ are $s_0(Y) = 1$ and $s_1(Y) = 2$. In this case \Cref{conj: sectional} holds.
    
    Conversely, we can consider the polar degrees of $\overline{Y} = \VV(y_1^2 - y_2 y_0)$ and compute that $\delta_0(\overline{Y}) = 2$ and $\delta_1(\overline{Y}) = 2$. This provides an example where the sum of the polar degrees of $\overline{Y}$ is a strict upper bound on the ED degree of $Y$ but the sum of the sectional degrees is exact. 
\end{example}
\end{example}

\section{Homogeneous equations for critical points}
\label{sec: homeq}

In this section, we define homogeneous critical point equations for the optimization problem \eqref{POP}.
We give two different sets of critical point equations for $\eqref{POP}$.
On the one hand, in \cite{ MR2507133} critical points are characterized 
as an intersection of the vanishing locus of homogeneous equations $\{\widetilde{f}_1 = \cdots = \widetilde{f}_m = 0\}$ with a projective determinantal variety $W$.
We generalize this approach by replacing projective space with an appropriate toric variety $X$.
On the other hand, we homogenize the Lagrange equations
$\mathbf{L}_{\mathrm{F}} = (f_1,\ldots,f_m,\ell_1,\ldots,\ell_n)$
in the Cox ring of a toric variety, $\P(\EE)$, which we introduce now. We show that both approaches define the desired critical point equations in \Cref{lemma: homogeneous criticality equations are correct} with \eqref{char1} concerning the former approach and \eqref{char2} the latter.

\subsection{Toric projective bundles} 
We now describe the toric structure on the projectivization of a direct sum of line bundles on toric varieties.
Let $X$ be a complete toric variety given by a fan $\Sigma$ and let $\EE=\LL_0\oplus\ldots\oplus \LL_m$ be a direct sum of line bundles on $X$. In this subsection, we will describe the fan of the total space of the projectivization $\P(\EE)$. For a more general study of line bundles with toric variety fiber see \cite{hofscheier2020cohomology}. 
\begin{lemma}\label{lem:projistoric}
    Let $X$ be a toric variety and let $\EE=\LL_0\oplus\ldots\oplus\LL_m$ 
    be a direct sum of line bundles. The total spaces of $\EE$ and $\P(\EE)$ can be given the structure of a toric variety. 
\end{lemma}
\begin{proof}
  For every line bundle $\LL_i$, there exists a torus invariant divisor $D_i$ such that $\LL_i$ is isomorphic to $\OO(D_i)$. Therefore, each line bundle $\LL_i$ on $X$ can be equipped with an equivariant structure, i.e. the action of $T$  on the total space of $\LL_i$ where $T$ is the torus acting on $X$. This makes the projection map equivariant.
    
    By fixing an equivariant structure on each of the line bundles $\LL_i$, we obtain a $T$-action on the total space of $\EE$. Finally, we extend the $T$-action on $\EE$ to the action of $T\times (\C^*)^{m+1}$ by making the second component act fiberwise in a natural way. This action is faithful and has an open-dense orbit in the total space of $\EE$. 

Moreover, the action of $T\times (\C^*)^{m+1}$ on $\EE$ descends to an action on $\P(\EE)$. The latter action has a one-dimensional kernel given by the diagonal subtorus in $(\C^*)^{m+1}$. Hence $\P(\EE)$ has the structure of a toric variety with respect to the factor torus
\begin{equation*}
        T\times \left((\C^*)^{m+1}/\C^*\cdot(1,\ldots,1)\right). \qedhere
\end{equation*}
\end{proof}

\begin{remark}
    Note that the divisor $D_i$ is defined up to addition of the principal divisor $div(u)$ of character $u\in \Z^n$ 
    or, equivalently, any two equivariant structures on $\LL_i$ differ by the action of a character of $T$. Therefore, the toric varieties defined by different choices of $D_i$ are isomorphic.
\end{remark}

We conclude by describing the defining fan of the projectivized total space 
$\P(\EE)$, when the defining line bundles of $\EE$ are torus equivariant. More precisely, for $0 \leq j \leq m$, we denote $D_j$ as a torus invariant divisor such that $\LL_i=\OO(D_i)$.
Each divisor $D_j$ defines a cone-wise linear function 
\[
\psi_i\colon\R^n = |\Sigma| \to \R.
\]
Let $\Psi=(\psi_0,\ldots,\psi_m)\colon |\Sigma| \to \R^{m+1}$ be the corresponding piece-wise linear map. 

Denote $\widetilde \Sigma\subset \R^n\times \R^{m+1}$ as the fan that is obtained as the graph of the function $\Psi$. That is, $\widetilde \Sigma$ consists of cones $\widetilde \sigma$ where
\[
\widetilde \sigma = \{(x,\Psi(x)) \,|\, x\in \sigma\}, \text{ for } \sigma\in \Sigma.
\]

We now abuse notation and denote $\R^n_{\geq 0}$ as the fan supported on the positive orthant, whose cones are all of the form $\sigma_J = \{x \in \R^n_{\geq 0}, \ \mid \ \forall j \in J: x_j = 0\}, \ J \subseteq [n].$
The fan defining the total space of $\EE$ consists of cones
\[
\widetilde\sigma + \tau \text{ for } \sigma\in \Sigma, \tau\in \R^n_{\geq 0} \times \{0\}.\]
Similarly, the fan $F$ defining the total space of $\EE$ with the zero section removed is given by
\[
\widetilde\sigma + \tau \text{ for } \sigma\in \Sigma, \tau\in  \partial \R^n_{\geq 0},
\]
where $ \partial\R^n_{\geq 0} =  \R^n_{\geq 0} \setminus \{ \R^n_{> 0} \}$ denotes the fan consisting of all cones in $\R^n_{\geq 0}$, except for the one of dimension $n$.
Finally, let $\mathcal S_0\subset \EE$ be the image of the zero section of $\EE$.  Now $\mathcal S_0$ is a torus invariant subset of $\EE$ and thus the natural projection of $\EE\setminus \mathcal S_0\to \P(\EE)$ is a toric morphism. 
On the level of fans, consider the projection 
\[
\tau \colon \R^n\times\R^{m+1} \longrightarrow \R^n\times\left(\R^{m+1}  /   \R\cdot(1,\ldots,1)\right)  \cong \R^n\times \R^{m} .
\]

Now let $S$ be the Cox ring of $X$, and let $S_\EE$ be the Cox ring of $\P(\EE)$.
By the above discussion, each ray of $\widetilde{\Sigma}$ is either of the form $\widetilde{\rho}$, where $\rho$ is a ray of $\Sigma$, or of the form $\{0\}\times e_i$, $i = 1, \dots, m+1$.
This splits the generators of $S_\EE$ over $\C$ into two groups.
The first group of generators $x_{\widetilde{\rho}}$ is bijective to the generators $x_\rho$ of $S$. We denote
members of the second group by
$\lambda_i = x_{ \{0\}\times e_{i+1}}$ and obtain the following proposition.
\begin{proposition}
\label{prop: Cox ring of P(E)}
    The Cox ring $S_\EE$ is isomorphic to 
    the free $S$-algebra $S[\lambda_0, \dots, \lambda_m]$.
\end{proposition}

\begin{remark}
    In the following, we are often in the situation where 
    $f$ is a global section of a torus invariant line bundle $\OO_X(D)$ on $X$.
    Now $\widetilde{f}$ denotes an element of $S \subseteq S_\EE$.
    At the same time, $f$ can be identified with a section of the bundle $\pi^*\OO_X(D)$ on $\P(\EE)$, where $\pi: \P(\EE) \longrightarrow X$ is the natural projection. When homogenizing, this gives rise to another element $\widetilde{f} \in S_\EE$.
    Direct computation shows that there is no need for disambiguation since both expressions are equal.
\end{remark}

\subsection{Constructing critical point equations in Cox rings}

\label{subsection: homogeneous critical point equations}
We start by fixing some notation and definitions.
For the rest of this section, let $\mathbf{F} =(f_0,\ldots,f_m)$ be a generic sparse system of polynomials 
in $\mathbb{C}[y_1,\ldots,y_n]$
with admissible support $\AA = (\AA_0,\ldots, \AA_m)$. 
Furthermore, $X$ denotes a toric variety that is appropriate for $\AA$, with fan $\Sigma$.

\begin{remark}
    Note that, since $\Sigma$ contains the positive orthant $\R_{\geq 0}^n$ as a cone, there is a distinct copy of the affine space $\C^n$ contained in $X$.
    For clarity, 
    we denote the affine variables in the coordinate ring $\mathbb{C}[y_1,\ldots,y_n]$ of $\C^n$ as $y_1,\ldots,y_n$, while the generators of the Cox ring $S$ of $X$ as $x_\rho$.
    By slight abuse of notation, we will denote the element $x_{e_j}$ in $S$ by $x_j$ for each $j = 1 \dots, n$.
\end{remark}

For every $i = 0, \dots, m$ let $\LL_i = \OO(-D_{f_i})$ denote the dual line bundle associated with $D_{f_i}$. Here $D_{f_i}$ is the torus invariant Weyl divisor on $X$, corresponding to the Newton polytope $\newt(f_i)$:
\begin{equation}
\label{eq: Weyl divisor Dfi}
D_{f_i} =\sum_{\rho \in \Sigma(1)} a_{\rho,i} D_\rho,\quad \text{where}\quad a_{\rho,i} = - \min \{ \langle m, \rho \rangle \ : \  m \in \newt(f_i)  \}.    
\end{equation}
We denote by $\EE$ the vector bundle 
$\EE = \LL_0 \oplus \cdots \oplus \LL_m$ with projectivized total space
 \[\P(\EE) = \{ (x, [\lambda]) \mid \ x\in X, \  [\lambda] \in \P \left(\EE(x)\right)  \}  .\]

The rest of this section is devoted to giving two different, but related, systems of homogeneous critical point equations for \eqref{POP}, one in the Cox ring $S$ of $X$, and one in the Cox ring $S_\EE$ of $\P(\EE)$.
On the one hand, critical points are characterized by the vanishing of
the Lagrange system $\mathbf{L}_{\mathrm{F}}$.
It describes the intersection of 
the
incidence variety
\[
Z^\circ := \{ (x,[\lambda]) \in \C^n\times \P^n :  \ \left( \nabla f_0 \mid \dots \mid \nabla f_m  \right) \lambda = 0   \}
\]
with the vanishing locus of $f_1, \dots, f_m$.
On the other hand, 
critical points are
characterized by the Jacobian $\left( \nabla f_0 \mid \dots \mid \nabla f_m  \right)$ dropping rank.
They form the intersection $V^\circ \cap W^\circ$, where
$V^\circ  \coloneqq \VV \left( f_1, \dots, f_m \right) \subseteq \C^n$, and
$W^\circ$ is the determinantal variety
\[W^\circ := \{ x \in \C^n \ : \  \rank \left( \nabla f_0 \mid \dots \mid \nabla f_m  \right) \leq m \}.\]

We proceed by giving homogeneous equations for
$V^\circ, W^\circ$ and $Z^\circ$ in $S$ and $S_\EE$ respectively.
Every polynomial $f_i$ is a global section of the line bundle $\OO_X(D_{f_i})$,
and its homogeneous form can be written as
\[
\widetilde{f_i} = \sum_{m \in \newt(f_i) \cap \Z^n}   c_{m,i} \prod_{\rho \in \Sigma(1)}  
x_\rho^{\langle m, \rho \rangle + a_{\rho,i}}.
\]
Here we homogenize $f_i$ as 
in \eqref{eq: homogenization} in  \Cref{sec: toric}. In particular, $\widetilde{f}_i$ is defined by our choice of line bundle $\OO_X(D_{f_i})$.

We denote $V$ as the closure of $V^\circ = \VV \left( f_1, \dots, f_m \right)$ in $X$.    
Observe that by genericity of $\mathbf{F}$, $V$ is equal to the vanishing locus of the homogeneous equations
$
V = \VV \left( \widetilde{f}_1, \dots \widetilde{f}_m \right).
$

We denote $M$ as a homogeneous version of the Jacobian matrix:
\[
M = \left( \widetilde{\nabla} \widetilde{f}_0  \mid \dots \mid  \widetilde{\nabla} \widetilde{f}_m \right).
\]
Here $\widetilde{\nabla}$ denotes the vector
$(\frac{\partial}{\partial x_1}, \dots, \frac{\partial}{\partial x_n})^T$. We use the notation
 $\widetilde{\nabla}$ instead of $\nabla$ to indicate that we differentiate with respect to coordinates in the Cox ring.
So $M$ has columns
$ \left( \frac{\partial}{\partial x_1} \widetilde{f}_i, \dots, \frac{\partial}{\partial x_n} \widetilde{f}_i  \right)^T$.
We define
\[
W \coloneqq \{ x \in X \ : \  \rank (M) \leq m  \}
\]
to be the vanishing locus of the maximal minors of $M$.
Furthermore, we let 
\[
Z \coloneqq \{ (x, [\lambda]) \in \P(\EE) \ : \  M(x) \lambda = 0  \}
\]
be the associated incidence variety, contained in the projectivized total space $\P(\EE)$.

The rest of this Section is devoted to proving Lemma \ref{lemma: homogeneous criticality equations are correct}.
It shows that the homogeneous critical point equations agree with the affine ones when restricted to affine space $\C^n$.

We need an observation about differentiating homogeneous polynomials.
Let $D$ be a torus invariant Weyl divisor on $X$ (or on $\P(\EE)$), and $f$ a global section of $\OO_X(D)$.
Observe that for $j = 1, \dots, n$ the Newton polytope of the differential $y_j \frac{\partial}{\partial y_j} f $ is contained in the rational polytope $e_j + \partial_j \newt(f)$, and
in particular $ \frac{\partial}{\partial y_j} f $ is
a global section of the sheaf
$\frac{1}{y_j}\OO_X(D -  D_{ e_j})$ (or of the sheaf $\frac{1}{y_j}\OO_{\P(\EE)}(D -  D_{\widetilde{e}_j})$). 
Direct computation shows the following proposition.

\begin{proposition}
\label{prop: homogenization commutes with differentiation}
Homogenization and differentiation commute:
$ \widetilde{\frac{\partial}{\partial y_j} f  }  = 
\frac{\partial}{\partial x_j} \widetilde{f}$. 
\end{proposition}

We denote $\widetilde{ \Phi}_F $ as the homogenization of the Lagrangian, $\Phi$, in the Cox ring $S_\EE$.
This makes sense since $\P(\EE)$ defines a global section of the sheaf $\OO_{\P(\EE)}(D_{\Phi_F})$,
associated with the Cayley polytope
$\newt(\Phi_F) = \cay(\newt(f_0), \dots, \newt(f_m))$.
By \Cref{prop: homogenization commutes with differentiation}, each of the defining equations
\[
0 = 
\left( M(x) \lambda \right)_j = 
\lambda_0 \frac{\partial}{\partial x_j} \widetilde{f}_0+ \cdots +
\lambda_m \frac{\partial}{\partial x_j} 
\widetilde{f}_m = \frac{\partial}{\partial x_j} \widetilde {\Phi}_F(\lambda, x)
\]
of $Z$
is equal to the homogenization
$\widetilde{\ell_j}$
of $\ell_j =  \widetilde { \frac{\partial}{\partial y_j} \Phi_F}(\lambda, x)$.
Here $\widetilde{\ell_j}$ is considered as a global section of the sheaf
$\OO_{\P(\EE)}(D_{\Phi_F} - D_{ \widetilde{e}_j})$.

 We denote
    $\widetilde{\mathbf{L}}_F = (\widetilde{f}_1,\ldots,\widetilde{f}_m,\widetilde{\ell}_1,\ldots,\widetilde{\ell}_n)$ as the  homogenized Lagrange system.
On one hand, we observed above that $Z$ is equal to the vanishing locus of $\widetilde{\ell}_1,\ldots,\widetilde{\ell}_n$.
On the other hand, the vanishing locus of $\widetilde{f}_1,\ldots,\widetilde{f}_m$ in $\P(\EE)$ is the preimage $\pi^{-1}(V)$
of the vanishing locus $V$ of $ \widetilde{f}_1,\ldots,\widetilde{f}_m $ in $X$.
We obtain the following Proposition.

\begin{proposition}
    The vanishing locus of \,$\widetilde{\mathbf{L}}_F$ in $\P(\EE)$ is the intersection
    $Z \cap \pi^{-1}(V)$.
\end{proposition}

The following lemma shows that the homogeneous critical point equations introduced in this chapter restrict on $\C^n$
to the expected affine critical point equations.
\begin{lemma}
    \label{lemma: homogeneous criticality equations are correct}
    The following three equalities hold:
    \begin{align}
    &V \cap \C^n =  V^\circ, \quad  \ W \cap \C^n =  W^\circ \label{char1} \\
    &Z \cap 
    \pi^{-1}(\C^n)
    =  Z^\circ, \label{char2}
    \end{align}
  where the intersection in \eqref{char1} is in $X$ and the intersection in \eqref{char2} is in $\P(\EE)$.  
\end{lemma}

\begin{proof}
The first of the equalities is clear, by the definition of $V$ as the closure of 
$V^\circ$.
To see the second equality,
we prove that the entries of $M$ are homogenizations of the entries of the Jacobian $\left( \nabla {f}_0, \dots, \nabla {f}_m \right)$.
This is a direct consequence of \Cref{prop: homogenization commutes with differentiation}, since for every $i=0, \dots, m$ and $j = 1, \dots, n$,
$ \widetilde{\frac{\partial}{\partial y_j} f_i  }  = 
\frac{\partial}{\partial x_j} \widetilde{f}_i$.
The third equality is analogous, since homogenizing the defining equations ${\ell}_1, \dots {\ell}_n$ of $Z^\circ$ yields the defining equations $\widetilde{\ell}_1, \dots \widetilde{\ell}_n$ of $Z$.
\end{proof}

We close this section with the following generalization of Euler's equation.
\begin{proposition}
\label{prop:Euler equation}
Let $D = \sum_{\rho \in \Sigma(1)} a_\rho D_\rho$ be a torus invariant Weyl divisor,
$f$ a global section of $\mathcal{O}_X (D)$ and $\tau \in \Sigma(1)$ a ray.
Then the generalized Euler equation
\begin{equation}
\label{eq:Euler formula}
-x_\tau \frac{\partial}{\partial x_\tau} \widetilde{f} + \tau_1 x_1\frac{\partial}{\partial x_1 } \widetilde{f} + \cdots + \tau_n x_n\frac{\partial}{\partial x_n } \widetilde{f}  = -a_\tau \widetilde{f}
\end{equation}
holds for the homogenization $\widetilde{f} \in S$ of $f$.
\end{proposition}
\begin{proof}
Equation \eqref{eq: homogenization}
reads
$
\widetilde{f} = \sum_{ m \in \Z^n  }
c_m\prod_{\rho \in \Sigma(1)} x_\rho^{\langle m, \rho \rangle + a_\rho}
$
and we have
\begin{align*}
&
-x_\tau \frac{\partial}{\partial x_\tau} \widetilde{f} + \tau_1 x_1\frac{\partial}{\partial x_1 } \widetilde{f} + \cdots + \tau_n x_n\frac{\partial}{\partial x_n } \widetilde{f} \\
=
&
\sum_{ m \in \Z^n  }
c_m
\left(-x_\tau \frac{\partial}{\partial x_\tau}
+ \tau_1 x_1\frac{\partial}{\partial x_1 }
+ \cdots+
x_n\frac{\partial}{\partial x_n }
\right)
\prod_{\rho \in \Sigma(1)}
x_\rho^{\langle m, \rho \rangle + a_\rho}
\\
=
&
\sum_{ m \in \Z^n  }
c_m
( -\langle m, \tau \rangle - a_\tau 
+ m_1 + a_{e_1} + \cdots 
+ m_n + a_{e_n}
)
\prod_{\rho \in \Sigma(1)}
x_\rho^{\langle m, \rho \rangle + a_\rho}
\\
=
&
\sum_{ m \in \Z^n  }
c_m
(- a_\tau
)
\prod_{\rho \in \Sigma(1)}
x_\rho^{\langle m, \rho \rangle + a_\rho}
= - a_\tau \widetilde{f}. \qedhere
\end{align*}
\end{proof}

\section{Computing the number of critical points}
\label{sec: secion_with_proofs_for_thm_A_and_B}

In this section we prove \Cref{thm: A} and \Cref{thm: C}, relying on the results from \Cref{sec: homeq}.
In \Cref{lemma: homogeneous criticality equations are correct} we characterized critical points of \eqref{POP}
in two ways. First, in \eqref{char1} we characterized the critical points of \eqref{POP} as an intersection $V \cap W$ in $X$.
Second,  we characterized the critical points of \eqref{POP} in \eqref{char2} by means of homogenized Lagrange equations
$\widetilde{\mathbf{L}}_F$ in the Cox ring of $\P(\EE)$. 
In this section, we show that all intersections are transversal and happen in $\C^n$.
This characterizes the number of critical points
as products of cohomology classes.
In the case of \Cref{thm: A} this product is a mixed volume.
The proof of \Cref{thm: C} relies on a characterization of $[W]$
as a Porteous class.

Before diving into the proofs, we wish to highlight where the assumptions of our main theorem are used.
We  use the assumption that $\AA$ be admissible in the proof of
\Cref{prop: gradient does not vanish uniformly}.
For the proof of \Cref{prop: rank M is m+1} we need that the closure $\VV$ of the constraint locus $\VV(f_1, \dots, f_m)$ is smooth.
This is guaranteed by the stronger assumption that  $\AA$ is strongly admissible in the proof of \Cref{thm: A} and \Cref{thm: B}.
For \Cref{thm: C} we assume $X$ is smooth in order to employ Porteous' formula.

\subsection{Preliminary results}
We start by proving some technical statements that are needed for the desired transversality results.
For the rest of the section we again fix the assumptions from \Cref{subsection: homogeneous critical point equations},
and assume that $V$ does not intersect the singular locus of $X$.
It follows from
Bertini's Theorem, that $V$ is smooth, which is the motivation for 
\Cref{def: strongly admissible point configuration} and 
\Cref{def: appropriate}.

For the next proposition, we use the following notion. Similar to the projection $\C^{n+1}\setminus\{ 0\} \to \P^n$, there exists the open subset $U_\Sigma \subseteq \C^{\Sigma(1)}$ with a projection $\tau:U_\Sigma \to X$. For a subvariety $Z$  of $X$, we define the cone $C(Z)$ over $Z$ to be the closure of the preimage $\tau^{-1}(Z)$ in $\C^{\Sigma(1)}$. The intersection $C(Z)\cap U_\Sigma$ forms a torus principal bundle over $Z$. In particular, $C(Z)\cap U_\Sigma$ is smooth if $Z$ is smooth.

\begin{proposition}
\label{prop: rank M is m+1}
The matrices
$\left( \widetilde{\nabla} \widetilde{f}_0, \dots, \widetilde{\nabla} \widetilde{f}_m \right)$ and
$\left( \widetilde{\nabla} \widetilde{f}_1, \dots, \widetilde{\nabla} \widetilde{f}_m \right)$ have full ranks $m+1$ and $m$ everywhere on $\VV \left( \widetilde{f}_0, \dots \widetilde{f}_m \right)$ and $\VV \left( \widetilde{f}_1, \dots \widetilde{f}_m \right)$ respectively.
\end{proposition}

\begin{proof}
The proof for the second matrix
is analogous, so
we only present the proof~for
\[
M = \left( \widetilde{\nabla} \widetilde{f}_0, \dots, \widetilde{\nabla} \widetilde{f}_m \right).
\]

Let $x \in \VV \left( \widetilde{f}_0, \dots \widetilde{f}_m \right)$ be arbitrary and $\sigma \in \Sigma$ the unique cone such that $x$ is contained in the torus orbit $O(\sigma)$.
Let $\widetilde{M}$ denote the matrix with rows
\begin{equation}
\label{eq: row of Msigma}
\left(
\frac{\partial}{\partial x_\rho} \widetilde{f}_0, \dots, 
\frac{\partial}{\partial x_\rho} 
\widetilde{f}_m
\right)
\end{equation}
for each ray $\rho$ in $\Sigma(1)$. 
The left kernel of $\widetilde{M}$
is
the tangent space of the cone
$C \left( \VV \left( \widetilde{f}_0, \dots ,\widetilde{f}_m \right) \right)$ in $\C^{\Sigma(1)}
$.
The Jacobian $\widetilde{M}_\sigma $, of the cone over the variety 
\[O(\sigma) \cap \VV \left( \widetilde{f}_0, \dots, \widetilde{f}_m \right)\]
is a submatrix of $\widetilde{M}$.
Its rows correspond to those rays
$\rho$ that are not contained in $\sigma$.
By our assumption at the beginning of this section, $V$ is disjoint from the singular locus of $X$, and we can apply Bertini's Theorem
to show that $\VV \left( \widetilde{f}_0, \dots , \widetilde{f}_m \right)$ is a smooth variety.
Furthermore, the intersection
$O(\sigma) \cap \VV \left( \widetilde{f}_0, \dots , \widetilde{f}_m \right)$ is transversal by \cite{khovanskii1978newton}, so $\widetilde{M}_\sigma$ is of full rank $m+1$ at $x$.
We now finish the proof by showing that the row span of $\widetilde{M}_\sigma$ is contained in the row span of $M$.
Let $\rho$ be any ray that is not contained in $\sigma$.
To show that the corresponding row
\eqref{eq: row of Msigma}
of $\widetilde{M}_\sigma$
is contained in the
row span of $M$,
we apply \Cref{prop:Euler equation} to all functions $\widetilde{f}_0, \dots, \widetilde{f}_m$.
The right side of equation $\eqref{eq:Euler formula}$ vanishes,
and we obtain
\[
\begin{pmatrix} \rho_1 x_1 \\ \vdots\\
\rho_n x_n
\end{pmatrix} ^T
M
 = x_\rho
 \begin{pmatrix}  \frac{\partial}{\partial x_\rho } \widetilde{f}_0 \\ \vdots\\
 \frac{\partial}{\partial x_\rho } \widetilde{f}_m
\end{pmatrix}^T. \qedhere
\]
\end{proof}

\begin{proposition}
    \label{prop: gradient does not vanish uniformly}
 The gradient $\widetilde{\nabla} \widetilde{f}_0 = \left(  \frac{\partial}{\partial x_j}\widetilde{f}_0 \right)_{j = 1, \dots, n}$ does not vanish on any orbit.
\end{proposition}
\begin{proof}
Towards a contradiction, we assume that there exists a cone $\sigma \in \Sigma$ such that 
for every $j = 1, \dots, n$ the polynomial $\frac{\partial}{\partial x_j}\widetilde{f}_0$ vanishes on the associated torus orbit $O(\sigma)$ of $X$.
We denote $\left.\frac{\partial}{\partial x_j}\widetilde{f}_0\right|_{O(\sigma)}$ as the restriction of $\frac{\partial}{\partial x_j}\widetilde{f}_0$ to the cone over $O(\sigma)$. It is obtained by substituting all variables $x_\rho$ with zero, where $\rho$ is contained in $\sigma$.

Now consider the face $\newt(f_0)^{\sigma}$
of $\newt(f_0)$ exposed by $\sigma$.
For every lattice point $m$ of $\newt(f_0)^{\sigma}$ the monomial 
\[
\frac{\partial}{\partial x_j} \prod_{ \rho \in \Sigma(1)} x_\rho^{\langle m, \rho \rangle + a_{\rho,0}},
\quad  a_{\rho,0} = - \min \{ \langle m, \rho \rangle \ : \  m \in \newt(f_0)  \}
\] of $\frac{\partial}{\partial x_j}\widetilde{f}_0$ only vanishes on $O(\sigma)$ if $m_j = 0$.
In particular, the face $\newt(f_0)^{\sigma}$ can only contain
the single element $\underline{0}$. 
By assumption~\ref{def: appropriate} on $X$, $\underline{0}$ is a smooth vertex of $\newt(f_0)$, and dual to the cone $\R_{\geq 0}^n$. Since $\sigma$ reveals the vertex $\underline{0}$, it has to intersect the interior of the positive orthant $\R_{\geq 0}^n$ and in fact
both cones are equal.
This leaves us with the case where the torus orbit is $\{0\}$.
But the gradient $\widetilde{\nabla} \widetilde{f}_0$  does not vanish uniformly at $0$. \end{proof}

Let $V$ and $W$ denote the varieties from \Cref{subsection: homogeneous critical point equations}:

\[V = \VV \left( \widetilde{f}_1, \dots \widetilde{f}_m \right) \quad \text{and} \quad 
W = \{ x \in X \ : \  \rank \left( \widetilde{\nabla} \widetilde{f}_0  \mid \dots \mid  \widetilde{\nabla} \widetilde{f}_m \right) \leq m  \}. \]
\begin{proposition}
The variety $V \cap W$ is of dimension zero.
\label{prop: V cap W is finite}
\end{proposition}
\begin{proof}
Towards a contradiction we assume that
there exists a torus orbit $O(\sigma)$,
and a curve $C$ such that $C$ is contained in the intersection $W \cap V \cap O(\sigma)$.
We denote by $\left.\widetilde{f}_0\right|_{O(\sigma)}$ the restriction of $\widetilde{f}_0$ to $O(\sigma)$. It is obtained by substituting all variables $x_\rho$ with zero, where $\rho$ is contained in $\sigma$.
We now distinguish two cases:
either $\left.\widetilde{f}_0\right|_{O(\sigma)}$ vanishes somewhere on $O(\sigma)$, or it is a scalar multiple of a monomial. In the first case $\left.\widetilde{f}_0\right|_{O(\sigma)}$ vanishes on $C$ by genericity. In particular, the matrix $M$ drops rank somewhere on the vanishing locus
$\VV \left( \widetilde{f}_0, \dots \widetilde{f}_m \right)$, contradicting \Cref{prop: rank M is m+1}.

In the second case, we now derive a contradiction from \Cref{prop: rank M is m+1} by showing that 
the matrix $\left( \widetilde{\nabla} \widetilde{f}_1, \dots, \widetilde{\nabla} ,\widetilde{f}_m \right)$
drops rank somewhere on $C$.
Suppose $\left.\widetilde{f}_0\right|_{O(\sigma)}$ is
a monomial.
Then each restriction
$\left.\frac{\partial}{\partial x_j}\widetilde{f}_0\right|_{O(\sigma)}$
is either a monomial or zero, and
by \Cref{prop: gradient does not vanish uniformly} there is an index 
$\ell = 1, \dots, n$ such that
$\left.\frac{\partial}{\partial x_\ell}\widetilde{f}_0\right|_{O(\sigma)}$ is not zero.
Without loss of generality, we assume 
$\ell=1$. 
Consider the following matrix, $M^*$, obtained by
subtracting 
from the $j-$th row of $M$ the multiple
\[
\frac{\frac{\partial}{\partial x_j}\widetilde{f}_0}{\frac{\partial}{\partial x_1}\widetilde{f}_0}
\left( \frac{\partial}{\partial x_1}\widetilde{f}_0, \dots , \frac{\partial}{\partial x_1} \widetilde{f}_m\right)  \]
of the first row, for each $j = 2, \dots, n$.
This eliminates the first entry in all but the first row:
\begin{gather*}
M^* = 
    \begin{bmatrix}
    \frac{\partial}{\partial x_1}\widetilde{f}_0 &  \frac{\partial}{\partial x_1} \widetilde{f}_1\dots  \frac{\partial}{\partial x_1}\widetilde{f}_m \\
        \begin{matrix}
        0 \\ \vdots \\ 0
    \end{matrix} 
    & A
\end{bmatrix}.
\end{gather*}
Since $M$ drops rank everywhere on $C$ and $\frac{\partial}{\partial x_1} \widetilde f_0$ is not identically zero, $A$ also drops rank on $C$.
Let $\mu = \left( \mu_1, \dots, \mu_n \right)^T$ be a vector of rational functions on $O(\sigma)$
satisfying
\[
A  \mu = 0
\]
everywhere on $C$.
Since the expression $ \mu_1 \frac{\partial}{\partial x_1} \widetilde{f}_1 + \dots + \mu_n \frac{\partial}{\partial x_n} \widetilde{f}_1$,
is not a monomial on $O(\sigma)$, it vanishes at some point $x$ in $C$ by genericity.
This shows that $\mu(x)$ is in the right kernel of 
$\left( \widetilde{\nabla} \widetilde{f}_1(x), \dots, \widetilde{\nabla} \widetilde{f}_m(x) \right)$, 
finishing the proof.
\end{proof}

\begin{lemma}
\label{lemma: transversality upstairs}
The intersection $Z \cap \VV(\widetilde{f}_1, \dots, \widetilde{f}_m)$ is transversal and contained in the big torus $(\C^*)^{n+m}$ in the toric variety $\P(\EE)$.
\end{lemma}

\begin{proof}
The image of $Z\cap \VV(\widetilde{f}_1, \dots, \widetilde{f}_m)$ under the natural projection $\pi: \P(\EE) \longrightarrow X$ is $W\cap \VV(\widetilde{f}_1, \dots, \widetilde{f}_m)$, which by \Cref{prop: V cap W is finite} is finite.
We prove below that $\pi$ bijectively identifies both sets.
In particular, the $n$ defining equations of $Z$, given by $M(x) \lambda = 0$, form a complete intersection when restricted to $\VV(\widetilde{f}_1, \dots, \widetilde{f}_m)$.

To inductively apply Bertini to the equations $\left(M(x) \lambda\right)_j = 0$ we now show that, for varying coefficients of $f_0$,
$Z \cap \VV(\widetilde{f}_1, \dots, \widetilde{f}_m)$ defines a basepoint-free family of varieties on the vanishing locus
$\VV(\widetilde{f}_1, \dots, \widetilde{f}_m)$ in $\P(\EE)$.
To do this, we fix any element $x$ of $V$ and show that $Z$ does not have a fixed point in the fiber $\pi^{-1}(x)$.
By \Cref{prop: rank M is m+1} 
the last $m$ columns 
$
\left( \widetilde{\nabla} \widetilde{f}_1, \dots, \widetilde{\nabla} \widetilde{f}_m \right)
$
of $M$ are linearly independent.
In particular,
varying the first column $\widetilde{\nabla} \widetilde{f}_0$ changes the unique solution $[\lambda]$ to $M(x) \lambda = 0$.
It now suffices to see that the gradient $\widetilde{\nabla} \widetilde{f}_0$
does not vanish uniformly at $x$, which by \Cref{prop: gradient does not vanish uniformly} is true for generic coefficients of $f_0$.
To see that $Z \cap \VV(\widetilde{f}_1, \dots, \widetilde{f}_m) $ is contained in the big torus, $(\mathbb{C}^*)^{n+m}$, orbit, we apply the same Bertini type argument to show transversality of the intersection $ O \cap \VV\left(\widetilde{f}_1, \dots, \widetilde{f}_m\right)   \cap Z$. Here $O$ denotes any torus orbit on $\P(\EE)$. 
For dimensional reasons, this intersection can only be nonempty for the big torus $(\mathbb{C}^*)^{n+m}$ orbit. 
\end{proof}

\subsection{The proof of \Cref{thm: A}}

The idea behind the proof of \Cref{thm: A} is to study the system of homogenized Lagrange equations $\widetilde{\mathbf{L}}_F = (\widetilde{f}_1,\ldots,\widetilde{f}_m,\widetilde{\ell}_1,\ldots,\widetilde{\ell}_n)$. We show that it comprises global sections of $\Q$-Cartier divisors, that intersect transversely and away from infinity.
This expresses the number of solutions as a product of Chern classes, which is a mixed volume.

For the rest of 
this subsection, we impose the assumptions of \Cref{thm: A}.
Let $\Sigma$ be the normal fan of the Minkowski sum of the polytopes $\newt(f_0), \dots, \newt(f_n)$,
and let
$X$ be the associated normal toric variety.
Then the assumptions from the beginning of \Cref{subsection: homogeneous critical point equations} are fulfilled,
since $X$ is appropriate for $\AA$, and $V$ does not intersect the singular locus of $X$.

Again, let $\Phi_F$ denote the Lagrangian 
$\Phi_F(\lambda, y) = f_0 - \sum_{i=1}^m \lambda_i f_i$, 
and $\ell_j$ the partial differentials $\frac{\partial}{\partial y_j} \left( f_0 - \sum_{i=1}^m \lambda_i f_i \right)$ of
$\Phi_F$.
For each $j = 1, \dots, n$, we define the homogenization $\widetilde{\ell}_j$ of $\ell_j$ as a section of the divisor $D_{\Phi_F} -  D_{ \widetilde{e}_j}$ on $\P(\EE)$. We want to prove that this divisor is the torus invariant divisor associated with a translate of the rational polytope $\partial_j \newt(\Phi_F)$ but in order to do that we need a few propositions.

\begin{proposition}\label{prop:O(1)cay}
Let $X$ be a toric variety and $\EE = \bigoplus \LL_{P_i}$ be a direct sum of line bundles on $X$. Then the the relative $\OO(1)$ bundle of $\mathbb{P}(\EE)$ is represented by the Cayley polytope
$\cay(P_1,\ldots, P_m)$.
\end{proposition}
\begin{proof}
    The space of sections $H^0(\P(\EE),\OO(1))$ is canonically isomorphic to 
    \[H^0(X,\EE)= \bigoplus_i H^0(X,\LL_{P_i}).\]
    Moreover, $H^0(X,\LL_{P_i})$ is the weight space of $H^0(X,\EE)$ with respect to $(\C^*)^m$ acting fiber-wise on $\P(\EE)$ corresponding to the $i$-th basis vector of $\Z^m$.
    Since the weights of the base torus acting on $H^0(X,\LL_{P_i})$ is given by the lattice points of $P_i$, we obtain the~result.
\end{proof}

\begin{proposition}
\label{prop: taking faces commutes with Cayley construction}
     Let $\sigma$ be a cone 
    in the normal fan $\Sigma(P_1+\cdots+P_n)$. Then the face of $\cay(P_1, \dots, P_n)$ exposed by $\widetilde{\sigma}$ is equal to the Cayley polytope of the faces
    $P_1^\sigma, \dots, P_n^\sigma$:
    \begin{equation}
    \label{eq: Cayley construction commutes with taking faces}
           \cay(P_1, \dots, P_n)^{\widetilde{\sigma}} = \cay(P_1^\sigma, \dots, P_n^\sigma).  
    \end{equation}
\end{proposition}

\begin{proof}
This can be done by direct computation. A different argument relies on \Cref{prop:O(1)cay}.
For this, denote by
$X_\sigma$ the closure of the torus orbit of $X_\Sigma$ corresponding to $\sigma\in \Sigma$.
By \Cref{prop:O(1)cay}, the equation \eqref{eq: Cayley construction commutes with taking faces} is equivalent to
\[
\OO_{\P(\EE)}(1)\big|_{X_\Sigma}
=
\OO_{\P(\EE|_{X_\Sigma})}(1). \qedhere
\]
\end{proof}

\begin{lemma}
          \label{lemma: differentials are basepoint-free}
    For all $j = 1, \dots, n$,
    the divisor $D_{\Phi_F} -  D_{ \widetilde{e}_j}$
    on 
    $\P(\EE)$  
    is rationally equivalent to the divisor associated with the rational polytope 
    $\partial_j \newt(\Phi_F)$.
\end{lemma}
  
\begin{proof}
We now prove that 
the divisor $D_{\Phi_F} -  D_{ \widetilde{e}_j}$ is associated with the polytope $  e_j+ \partial_j \newt(\Phi_F)$. Note that $ e_j+ \partial_j \newt(\Phi_F)$ is the intersection of $\newt(\Phi_F)$ with the affine halfspace $\{ x_{j} \geq 1 \}$. 
We have to prove that the support function of
$\newt(\Phi_F) \cap{ \{x_j \geq 1 \} }$ takes the same value on all rays of $\Sigma_\EE$, except for $\widetilde{e_j}$,
where it differs by one.
Let $v$ be any element of $\R^{n+m}$.
The value
\[
- \min \{ \langle w, v\rangle \ : \  w \in \newt(\Phi_F) \}
\]
of the support function of $\newt(\Phi_F)$ on $v$
can only differ if the face $\newt(\Phi_F)^{v}$ is contained in the facet
\[\newt(\Phi_F)^{\widetilde e_j} = \newt(\Phi_F) \cap{ \{ x_j = 0 \} }.\]
Note that a face of the form $\newt (\Phi_F)^{  \{0\} \times e_i  } $ is equal to the Cayley polytope
$$\cay(\AA_0, \dots, \AA_{i-1}, \AA_{i+1}, \dots, \AA_m),$$ where we omit one of the constraints.
In particular, it is always a facet, so
we may restrict to rays of the form $\widetilde{\rho}.$
Let  $\widetilde{\rho}$ be a ray such that 
$\newt(\Phi_F)^{\widetilde \rho}$ is contained in $\newt(\Phi_F)^{\widetilde e_j}$.
By \Cref{prop: taking faces commutes with Cayley construction}  we have
\begin{align*}
    & \,\cay \left( \newt(f_0)^{\rho}, \dots, \newt(f_m)^{\rho} \right) \\
 = 
 &\,\newt(\Phi_F)^{\widetilde{\rho}} \\
  \subseteq
 & \,\newt(\Phi_F)^{\widetilde{e_j}} \\
=
& \,\cay \left( \newt(f_0)^{e_j}, \dots, \newt(f_m)^{e_j} \right),
\end{align*}
implying $ \newt(f_i)^{\rho} \subseteq \newt(f_i)^{e_j}$ for all $i =  0, \dots, m$.
We obtain
\[
\left( \sum_{i = 0}^m  \newt(f_i) \right )^{\rho}
 =  \sum_{i = 0}^m  \newt(f_i) ^{\rho}
 \subseteq \sum_{i = 0}^m  \newt(f_i) ^{e_j}
 =  \left ( \sum_{i = 0}^m \newt(f_i) \right )^{e_j}.
\]
This is an inclusion of facets of the Minkowski sum $\sum_{i = 0}^m \newt(f_i)$, 
so $\rho = e_j$.
\end{proof}

Before proving \Cref{thm: A} we need to prove a statement about the intersection of $\Q$-Cartier divisors.
\begin{lemma}[Generic intersection of $\mathbb  Q$-Cartier divisors]
\label{lemma: intersection of Q-cartier divisors}
Let $X$ be a normal, proper variety of dimension $n$ with Weyl divisors $D_{1}, \dots, D_{n}$, and let $k$ be an integer such that $\OO(k D_{i})$ is a line-bundle for each $i = 1, \dots, n$.
Let $\widetilde{f_i}$ be a global section of $\OO(D_{i})$ for $i=1, \dots, n$ such that
$\VV(\widetilde{f}_1, \dots, \widetilde{f}_n)$ is a zero-dimensional smooth scheme contained in the smooth locus of $X$.
Then
\[
k^{n} \# \VV(\widetilde{f}_1, \dots, \widetilde{f}_n) = 
c_1(\OO(k D_{1})) \cdots c_1(\OO(k D_{n})).
\]
\end{lemma}
\begin{proof}
The length of the zero-dimensional scheme $\VV((\widetilde{f}_1)^k, \dots, (\widetilde{f}_n)^k)$is equal to the product
$c_1(\OO(k D_{1})) \cdots c_1(\OO(k D_{n}))$ of Chern classes.
On the other hand, since $\widetilde{f}_1, \dots, \widetilde{f}_n$ intersect transversely, each isolated point of 
$\VV((\widetilde{f}_1)^k, \dots, (\widetilde{f}_n)^k)$ is isomorphic
to the scheme 
$\operatorname{Spec}  \C[X_1, \dots, X_n] / \langle X_1^k, \dots, X_n^k \rangle$.
In particular, we have
\[ 
k^{n} \cdot \operatorname{length}(\VV(\widetilde{f}_1, \dots, \widetilde{f}_m)) =
\operatorname{length}(\VV((\widetilde{f}_1)^k, \dots, (\widetilde{f}_n)^k)),
\]
finishing the proof.
\end{proof}

\begin{proof}[Proof of \Cref{thm: A}]
The vanishing locus of the homogenized system of Lagrange equations
$\widetilde{\mathbf{L}}_F = (\widetilde{f}_1,\ldots,\widetilde{f}_m,\widetilde{\ell}_1,\ldots,\widetilde{\ell}_n)$
 is the intersection of $Z$ with the vanishing locus of $\widetilde{f}_1,\ldots,\widetilde{f}_m$ and by \Cref{lemma: transversality upstairs} this is a smooth zero-dimensional variety, contained in $(\mathbb{C}^*)^{n+m}$.
By \Cref{lemma: homogeneous criticality equations are correct}, the algebraic degree of sparse polynomial optimization is equal to its cardinality.
According to \Cref{lemma: differentials are basepoint-free}, the system $\widetilde{\mathbf{L}}_F$ comprises global sections of $\Q$-Cartier divisors, associated with the respective, rational, polytopes $\newt(f_1), \dots, \newt(f_m), \partial_1  \newt( \Phi_F ) ,\dots, \partial_n  \newt(\Phi_F ).$
As a consequence of \Cref{lemma: intersection of Q-cartier divisors}, and using multilinearity of the mixed volume, 
we can express the number of solutions to $\widetilde{\mathbf{L}}_F$ as the mixed volume \eqref{eq:Mixed_Volume_partialDelta} of these polytopes.
\end{proof}

\subsection{The proof of \Cref{thm: C}}
In this section, we study the intersection of the determinantal variety $W$ with the vanishing locus $V$ of $\widetilde{f}_1,\ldots,\widetilde{f}_m$ in $X$.
The proof of \Cref{thm: C} rests on a proof of transversality, and a characterisation of the cohomology class $[W]$ as a Porteous class.

We start by recalling Porteous' formula, also called the Giambelli–Thom–Porteous formula.
For more details refer to chapter 14 in \cite{Fulton+1998} and chapter 12 in \cite{eisenbud_harris_2016}.
The following statement is a special case of Theorem 12.4 in \cite{eisenbud_harris_2016}.
\begin{theorem}(Porteous' formula)\label{thm: Porteous formula, general case}
Let 
$\varphi : \EE \longrightarrow \FF$ be a morphism of vector bundles
of ranks $m+1 \leq n$ on a smooth proper variety $X$ of dimension $n$.
We denote by $W$ the (possibly non reduced) degeneracy locus of $\varphi$,
supported on the set
\[
|W| = \{ x \in X \ : \  \varphi_x: \EE(x) \longrightarrow \FF(x) \textrm{ is not injective}  \}.
\]
If $W$ is pure of codimension $n-m$ then
the cohomology class of $W$ is the $n-m$ graded part of the product of the total Segre class $\ss(\EE )$ and the total Chern class~$\operatorname{c}( \FF ) $:
\[
\left[ W \right] = \left( \operatorname{s}( \EE ) \chern(\FF ) \right)_{n-m}.
\]

\end{theorem}
We need the following modified version of Porteous' formula which only requires
$W$ to be pure dimensional after restricting to a subvariety $\hat{X}$ of $X$.
\begin{cor}
\label{cor: Porteous when intersected with V}
Under the assumptions of Theorem~\ref{thm: Porteous formula, general case}, let $\hat{X}$ be an irreducible closed subvariety of $X$ of codimension $k$, which intersects $W$ transversely and let the intersection $\hat{X} \cap W$ be pure of codimension $n-m+k$.
Then the cohomology class $\left[ \hat{X} \cap W \right]$ is given by
\[
\left[ \hat{X} \cap W \right] =  \left[ \hat{X} \right] \cdot \left( \operatorname{s}( \EE ) \chern(\FF ) \right)_{n-m}.
\]
\end{cor}
\begin{proof}
Note that $\hat{X} \cap W$ is the degeneracy locus of the restriction $\left.\varphi\right|_{\hat{X}}$. By applying Porteous' formula to $\phi|_{\hat{X}}\colon \EE|_{\hat{X}} \to \FF|_{\hat{X}}$ we get
\[
[\hat{X} \cap W] = \left( \operatorname{s}( \left.\EE\right|_{\hat{X}} ) \chern(\left.\FF\right|_{\hat{X}} ) \right)_{n-m+k}.
\]
Lastly we notice that $\left( \operatorname{s}( \left.\EE\right|_{\hat{X}} ) \chern(\left.\FF\right|_{\hat{X}} ) \right)_{n-m+k} = \left[ \hat{X} \right ] \cdot\left( \operatorname{s}( \EE ) \chern(\FF ) \right)_{n-m}$ by naturality of characteristic classes.
\end{proof}

\begin{lemma}
\label{lemma: transversality downstairs}
Under the assumptions of \Cref{thm: C}
the intersection $V \cap W$ is transversal and contained in 
$\left(\C^*\right)^n$.
\end{lemma}
\begin{proof}
Under the assumptions of \Cref{thm: C}, the assumptions from \Cref{subsection: homogeneous critical point equations} are satisfied.
The inclusion in 
$(\mathbb{C}^*)^{n}$
follows from \Cref{lemma: transversality upstairs}.
We now show that 
transversality of the intersection $V \cap W$
follows from transversality of the intersection of
$Z$ with $\VV(\widetilde{f}_1, \dots, \widetilde{f}_m)$.
Let $\pi: \P(\EE) \longrightarrow X$ denote the natural projection and let $z = (x, [\lambda])$ be any element of $\P(\EE)$.
If $Z$ intersects $\pi^{-1}(V)$
transversely at $z$, then for the tangent spaces $T_{Z,z}$ and $T_{\pi^{-1}(V), z}$ at $z$ it holds
\begin{equation}
\label{eq: transversality}
    T_{Z,z} + T_{\pi^{-1}(V), z} = T_{\P(\EE),z}.
\end{equation}

To see that $W$ and $V$
intersect transversely at $x$ we show
$
T_{W, x} + T_{V, x} = T_{X, x}
$.
We apply the differential $d\pi$
to both sides of \eqref{eq: transversality} and note that we have the inclusions
\[
d\pi(T_{Z,z}) \subseteq T_{W, x}, \
d\pi (T_{\pi^{-1}(V), z}) \subseteq T_{V, x}, \
d\pi(T_{\P(\EE),z}) = T_{X, x}. \qedhere
\]
\end{proof}

\begin{proof}[Proof of \Cref{thm: C}]
By \Cref{lemma: homogeneous criticality equations are correct} the algebraic degree of sparse polynomial optimization
is the cardinality of $V \cap W \cap \C^n$.
By \Cref{lemma: transversality downstairs},
the scheme theoretic intersection $ V \cap W$ is a smooth variety of dimension zero, contained in 
$(\mathbb{C}^*)^n$.
We now finish the proof by verifying the assumptions of \Cref{cor: Porteous when intersected with V}.

Let $D_{f_i}$ be the Weyl divisors introduced in equation \eqref{eq: Weyl divisor Dfi}. By the assumptions of \Cref{thm: C}, $X$ is smooth. In particular, all divisors considered in this proof are Cartier. The variety $W$ is defined to be the degeneracy locus of the matrix $M$, whose entries are global sections of the bundle
$\OO_X(D_{f_i} -  D_{ e_j})$.
The transpose of $M$ defines a morphism $ \varphi : \EE \rightarrow \FF$ of vector bundles, where 
\[
\EE = \OO_X(-D_{f_0}) \oplus \cdots \oplus \OO_X(-D_{f_m}),  \text{ and }  \FF = \OO_X(-D_{e_1}) \oplus \cdots \oplus \OO_X(-D_{e_n}).
\]
Now $W$ is the degeneracy locus of $\varphi$, further $V \cap W$ is pure of dimension zero. Finally, $\LL_{\AA_i} = \OO_X(D_{f_i})$ which finishes the proof.
\end{proof}

\bibliographystyle{alpha}
\bibliography{adpo}
\end{document}